\documentclass[twoside]{amsart}
\usepackage{latexsym}
\usepackage{amssymb,amsmath,amsopn}
\usepackage[dvips]{graphicx}   
\usepackage{color,epsfig}      
\usepackage{verbatim}      

               {\begin{list}{}{\leftmargin#1\rightmargin#2}\item{}}%
               {\end{list}}

\newtheorem{thm}{Theorem}
\newtheorem{lem}{Lemma}
\newtheorem{prop}{Proposition}
\theoremstyle{definition}
\newtheorem{defn}{Definition}
\newtheorem{rem}{Remark}

\newtheorem{cor}{Corollary}
\renewcommand{\Re}{\mathbb R}

\renewcommand{\S}{\mathbb{S}}

\definecolor{green}{rgb}{0.0, 0.75, 0.25}

\def\eea{\end{eqnarray}}

\DeclareMathOperator{\aff}{aff}
\DeclareMathOperator{\relint}{relint}

\DeclareMathOperator{\inter}{int}
\DeclareMathOperator{\bd}{bd}
\DeclareMathOperator{\dist}{dist}
\DeclareMathOperator{\conv}{conv}

\DeclareMathOperator{\cl}{cl}

\begin{document}

\title[Balancing polyhedra]{Balancing polyhedra}
\author[G. Domokos, F. Kov\'acs, Z. L\'angi, K. Reg\H{o}s and P.T. Varga] {G\'abor Domokos, Fl\'ori\'an Kov\'acs, Zsolt L\'angi, Krisztina Reg\H{o}s and P\'eter T. Varga}
\address{G\'abor Domokos, MTA-BME Morphodynamics Research Group and Dept. of Mechanics, Materials and Structures, Budapest University of Technology,
M\H uegyetem rakpart 1-3., Budapest, Hungary, 1111}
\email{domokos@iit.bme.hu}
\address{Fl\'ori\'an Kov\'acs, Dept. of Structural Mechanics, Budapest University of Technology,
M\H uegyetem rakpart 1-3., Budapest, Hungary, 1111}
\email{kovacs.florian@epito.bme.hu}
\address{Zsolt L\'angi, MTA-BME Morphodynamics Research Group and Dept. of Geometry, Budapest University of Technology,
Egry J\'ozsef utca 1., Budapest, Hungary, 1111}
\email{zlangi@math.bme.hu}
\address{Krisztina Reg\H os, MTA-BME Morphodynamics Research Group,
M\H uegyetem rakpart 1-3., Budapest, Hungary, 1111}
\email{regoskriszti@gmail.com}
\address{P\'eter T. Varga, MTA-BME Morphodynamics Research Group,
M\H uegyetem rakpart 1-3., Budapest, Hungary, 1111}
\email{petercobbler@gmail.com}
\thanks{Support of the NKFIH Hungarian Research Fund grant 119245 and of grant BME FIKP-V\'IZ by EMMI is kindly acknowledged.
ZL has been supported by the Bolyai Fellowship of the Hungarian Academy of Sciences and partially supported by the UNKP-19-4 New National Excellence Program of the Ministry of Human Capacities. The authors thank Mr. Otto Albrecht for backing the prize for the complexity of the G\"omb\"oc-class.  
Any solution should be sent to the corresponding author as an accepted publication in a mathematics journal  of worldwide reputation and it must also have general acceptance in the mathematics community two years after.  The authors are indebted to Dr. Norbert Kriszti\'an
 Kov\'acs for his invaluable advice and help in printing the 9 tetrahedra and 7 pentahedra. }


\begin{abstract}
We define the mechanical complexity $C(P)$ of a convex polyhedron $P,$ interpreted as a homogeneous solid, as the difference between the  total number of its faces, edges and vertices and the number of its static equilibria, and
the  mechanical complexity $C(S,U)$ of primary equilibrium classes $(S,U)^E$ with $S$ stable and $U$ unstable equilibria as the infimum of the mechanical complexity of all polyhedra in that class.
We prove that the mechanical complexity of a class  $(S,U)^E$ with $S, U > 1$ is the minimum of $2(f+v-S-U)$ over all polyhedral pairs  $(f,v )$, where a pair of integers is called a polyhedral pair if there is a convex polyhedron with $f$ faces and $v$ vertices.
In particular, we prove that the mechanical complexity of a class $(S,U)^E$ is zero if, and only if there exists a convex polyhedron with $S$ faces and $U$ vertices.
We also give asymptotically sharp bounds for the mechanical complexity of the monostatic classes $(1,U)^E$ and $(S,1)^E$, and offer a complexity-dependent prize for the complexity of the G\"omb\"oc-class $(1,1)^E$. 
\end{abstract}
\maketitle

\tableofcontents

\section{Introduction}\label{sec:intro}

\subsection{Basic concepts and the main result}

Polyhedra may be regarded as purely geometric objects, however, they are also often intuitively identified with solids. Among
the most obvious sources of such intuition are dice which appear in various polyhedral shapes: while classical, cubic dice have 6 faces, a large diversity of other dice exist as well: dice with 2, 3, 4, 6, 8, 10, 12, 16, 20, 24, 30 and 100 faces appear in various games \cite{dice}. The key idea behind throwing dice is that each of the aforementioned faces is associated with a stable mechanical
equilibrium point where dice may be at rest on a horizontal plane. Dice are called \emph{fair} if the probabilities to rest on any face (after a random throw) are equal  \cite{Diaconis}, otherwise they are called \emph{loaded}
\cite{DawsonFinbow}. The concept of mechanical equilibrium may also be defined in purely geometric terms:
\begin{defn}\label{def1}
Let $P$ be a convex polyhedron, let $\inter P$ and $\bd P$ denote its interior and boundary, respectively and let $c \in \inter P$.  We say that $q \in \bd P$ is an \emph{equilibrium point} of $P$ with respect to $c$ if the plane $H$ through $q$ and perpendicular to $[c,q]$ supports $P$ at $q$. In this case $q$ is \emph{nondegenerate}, if $H \cap P$ is the (unique) face of $P$ that contains $q$ in its relative interior. A nondegenerate equilibrium point $q$ is called \emph{stable, saddle-type} or \emph{unstable}, if $\dim (H \cap P) = 2,1$ or $0$, respectively.
\end{defn}
Throughout this paper we deal only with equilibrium points with respect to the center of mass of polyhedra, assuming uniform density.
A \emph{support plane} is a generalization of the tangent plane for non-smooth objects. While it is a central concept of convex geometry its name may be related to the mechanical concept of equilibrium.
If $c$ coincides with the center of mass of $P$, then equilibrium points gain intuitive interpretation as locations on $\bd P$
where $P$ may be balanced if it is supported on a horizontal surface (identical to the support plane) without friction in the presence of uniform gravity.
Equilibrium points may belong to three stability types: faces may carry stable equilibria, vertices may carry unstable equilibria and edges may carry saddle-type equilibria.
Denoting their respective numbers by $S,U,H$, by the Poincar\'e-Hopf formula \cite{Milnor} for a convex  polyhedron one obtains the following relation for them:
\begin{equation}\label{Poincare}
S+U-H=2,
\end{equation}
which is strongly reminiscent of the well-known Euler formula
\begin{equation}\label{Euler}
f+v-e=2,
\end{equation}
relating the respective numbers $f$, $v$ and $e$ of the faces, vertices and edges of a convex polyhedron.
In the case of regular, homogeneous, cubic dice the formulae (\ref{Poincare}) and (\ref{Euler}) appear to express
the same fact, however, in case of irregular polyhedra the connection is much less apparent.
While the striking similarity between (\ref{Poincare}) and (\ref{Euler})
can only be fully explained via deep topological and analytic ideas \cite{Milnor}, our goal in this paper
is to demonstrate an interesting connection at an elementary, geometric level. To this end, we define
\begin{equation}\label{sum}
\begin{array}{rcl}
N & = & S+U+H, \\
n & = & f+v+e.
\end{array}
\end{equation}
Figure \ref{fig:1} shows three polyhedra where the values for all these quantities can be compared.

\begin{figure}[ht]
\begin{center}
\includegraphics[width=0.9\textwidth]{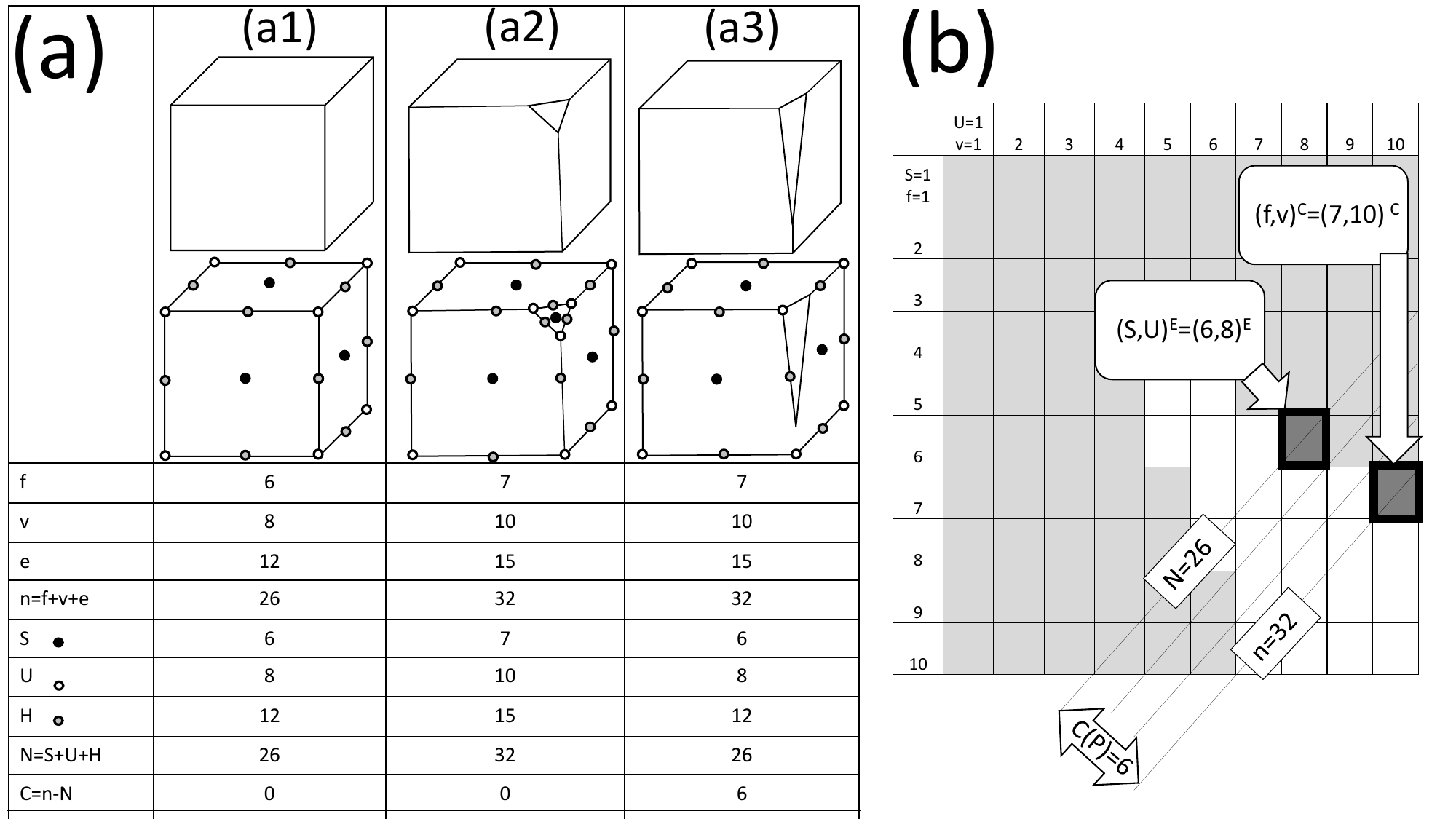}
\caption{ (a) Three polyhedra interpreted as homogeneous solids with given numbers for faces ($f$), 
vertices ($v$), edges ($e$), stable equilibria ($S$), unstable equilibria ($U$) and saddle-type equilibria ($H$), their respective sums $n=f+v+e$, $N=S+U+H$
and mechanical complexity $C=n-N$ (given in Definition \ref{defn:complex}). (b) Polyhedron in column $a3$ shown on the overlay of the $(S,U)$ and $(f,v)$ grids, complexity obtained from distance between corresponding diagonals.}
\label{fig:1}
\end{center}
\end{figure}
The numbers $S,U,H$ may serve, from the mechanical point of view, as a first-order characterization of $P$ and
via (\ref{Poincare}) the triplet $(S,U,H)$ may be uniquely represented by the pair $(S,U)$, which is called \emph{primary equilibrium class} of $P$ \cite{VarkonyiDomokos}.  Based on this, we denote by $(S,U)^E$ the family of all convex polyhedra having $S$ stable and $U$ unstable equilibrium points with respect to their centers of mass.
In an analogous manner, the numbers $(v,e,f)$ (also called the $f$-vector of $P$) serve as a first-order  combinatorial characterization of $P$, and via (\ref{Euler}) they may be uniquely represented by the pair $(f,v)$.  Here, we call the the family of all convex polyhedra having $v$ vertices and $f$ faces the \emph{primary combinatorial class} of $P$, and denote it by $(f,v)^C$. The face structure of a convex polyhedron $P$ permits a finer combinatorial description of $P$. In the literature, the family of convex polyhedra having the same face lattice is called a combinatorial class; here we call it a \emph{secondary combinatorial class}, and discuss it in Section~\ref{sec:conclusions}. In an entirely analogous manner, one can define also secondary equilibrium classes of convex bodies, for more details the interested reader is referred to \cite{DomokosLangiSzabo}.
While it is immediately clear that for any polyhedron $P$ we have
\begin{equation} \label{trivbound}
f \geq S, v \geq U,
\end{equation}
inverse type relationships (e.g. defining the minimal number of faces and vertices for given numbers of equilibria) are much less obvious.

A trivial necessary condition for any die to be fair can be stated as $f=S$ and
it is relatively easy to construct a polyhedron with this property. The opposite extreme case (when a polyhedron is stable only on one of its faces) appears to be far more complex and several papers \cite{Bezdek, Conway, Reshetov} are devoted to this subject to which we will return. Motivated by this intuition we define
the \emph{mechanical complexity} of polyhedra.
\begin{defn} \label{defn:complex}
Let $P$ be a convex polyhedron and let $N(P), n(P)$ denote the total number of its equilibria and the total number of its \emph{k-faces} (i. e., faces of $k$ dimensions) for all values $k=0,1,2$, respectively. Then
 $C(P)=n(P)-N(P)$ is called the mechanical complexity of $P$.
\end{defn} 
Mechanical complexity may not only be associated with individual polyhedra but also with primary equilibrium classes.
\begin{defn}\label{defn:class}
If  $(S,U)^E$ is a primary equilibrium class, then the quantity
\[
C(S,U)= \inf \{ C(P) : P \in  (S,U)^E \}
\]
is called the mechanical complexity of $(S,U)^E$.
\end{defn}
Our goal is to find the values of $C(S,U)$  for all primary equilibrium classes. For $S,U>1$ we will achieve this goal while for $S=1$ or $U=1$ we provide some partial results.
To formulate our main results, we introduce the following concept:
\begin{defn}\label{def:pair}
Let $x,y$ be positive integers. We say that  $( x,y )$ is a polyhedral pair if and only if $x \geq 4$ and $\frac{x}{2}+2 \leq y \leq 2x-4$.
\end{defn}
The  combinatorial classification of convex polyhedra was established by Steinitz \cite{Steinitz1, Steinitz2}, who proved, in particular, the following.
\begin{thm}\label{thm:steinitz}
For any positive integers $f,v$, there is a convex polyhedron $P$
with $f$ faces and $v$ vertices if and only if  $(f,v)$ is a polyhedral pair.
\end{thm}

\begin{rem}\label{rem:lowerbound}
Let  $(S,U)^E$ be a primary equilibrium class with $S, U \geq 1$, and let
\[
R(S,U) = \inf\{ f+v-S-U:  (f,v) \mbox{ is a polyhedral pair and $f$,$v$ satisfy } (\ref{trivbound})\}.
\]
The geometric interpretation of $R(S,U)$ is given in the left panel of
Figure \ref{fig:chart}. Since (\ref{trivbound}) holds for any polyhedron $P \in  (S,U)^E$, we immediately have the trivial lower bound for mechanical complexity:
\begin{equation}\label{eq:lowerbound}
C(S,U) \geq 2R(S,U).
\end{equation}
Based on Definition \ref{def:pair}, the function $R(S,U)$ can be expressed as
\begin{equation}\label{Rformula}
R(S,U)= \left\{
\begin{array}{rcl}
\lceil \frac{S}{2}\rceil -U+2, & \mbox{ if } & S > 4 \mbox{ and } S > 2U-4, \\
\lceil \frac{U}{2}\rceil -S+2, & \mbox{ if }  & U > 4 \mbox { and } U > 2S-4, \\
 8-S-U, &  \mbox { if }  & S,U \leq 4, \\
0 &  & \mbox{ otherwise.}
\end{array}
\right.
\end{equation}
           							
\end{rem}
\noindent Our main result is Theorem~\ref{thm:main1}, stating that this bound is sharp if $S,U>1$:
\begin{thm}\label{thm:main1}
Let $S,U \geq 2$ be positive integers. Then
C(S,U) = 2R(S,U).
\end{thm}
We remark that, as a consequence of Theorem \ref{thm:main1}, $C(S,U)=0$ if and only if $ (S,U)$ is a polyhedral pair.
For monostatic equilibrium classes ($S=1$ or $U=1$) we cannot provide a sharp value for their mechanical complexity. However, we will provide
an upper bound for their complexity, which differs from $2R(S,U)$ only by a constant:
\begin{thm}\label{thm:SUbounds}
If $S \geq 4$ then $C(S,1) \leq  59 + (-1)^S + 2 R(S,1)$; if $U \geq 4$ then $C(1,U) \leq 90+2 R(1,U)$.
\end{thm}
We also improve the lower bound (\ref{eq:lowerbound}) in some of these classes by generalizing a theorem of Conway \cite{Dawson} about the non-existence of a homogeneous tetrahedron with only one stable equilibrium point. We state our result in the following form:
\begin{thm}\label{thm:tetra}
Any homogeneous tetrahedron has $S \geq 2$ stable and $U \geq 2$ unstable equilibrium points.
\end{thm}
\noindent We summarize all results (including those about monostatic classes) in Figure~\ref{fig:chart}.

\begin{figure}[ht]
\begin{center}
\includegraphics[width=0.9\textwidth]{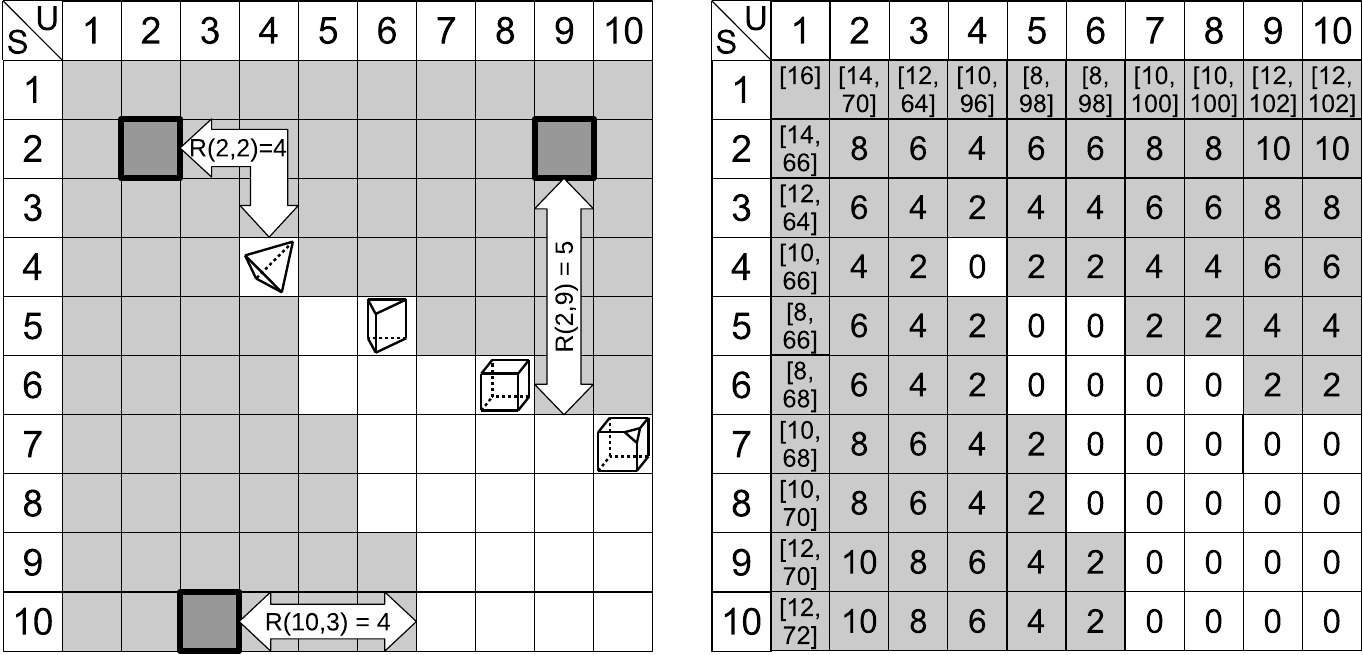}
\caption{Summary of results for $S,U \leq 10$. Left panel: the $(S,U)$ grid with some selected polyhedra as examples. Polyhedral pairs on the $(S,U)$ grid have white background. The function $R(S,U)$ illustrated for classes
$(2,2)^E, (2,9)^E, (10,3)^E.$ Right panel: Mechanical complexity of equilibrium classes 
 $(S,U)^E$.  Polyhedral pairs on the $(S,U)$ grid have white background. Sharp values
for mechanical complexity $C(S,U)$ are given as integers without brackets. In column $U=1$ and row $S=1$ we give bounds.
If two integers are given in square brackets then they are the lower and upper bounds for $C(S,U)$, if only one integer
is given in square brackets then it is the lower bound (and no upper bound is available).}
\label{fig:chart}
\end{center}
\end{figure}

\subsection{Sketch of the proof}
The main idea of the proofs of Theorems ~\ref{thm:main1} and \ref{thm:SUbounds} is to provide explicit constructions for at least one polyhedron $P$ in each class $ (S,U)^E, S,U>1$ with mechanical complexity $C(P)=2R(S,U)$,
in class $(S,1)^E, S\geq 4$ with $C(P)= 59+ (-1)^S + 2 R(S,1)$, and in class $(1,U)^E, U\geq 4$
with $C(P)= 90+2R(1,U)$.
By Definition \ref{defn:class}, such a construction establishes an upper bound for $C(S,U)$.
In case of $S>1$ and $U>1$, by Remark~\ref{rem:lowerbound}, this coincides with the lower bound while
for $S=1$ or $U=1$ the bounds remain separate.

Our proof consists of five parts:

(a) for classes  $(S,S)^E$ with $S\geq 4$, suitably chosen pyramids have zero mechanical complexity (Section \ref{sec:proof}).

(b) for classes  $1 < S\leq 5$ and $1<U \leq 5$, $(S,U)^E \neq (4,4)^E, (5,5)^E$, we provide examples found by computer search (Subsection \ref{subsec:nonpolyhedralpair}, Tables \ref{9classes} and \ref{6classes}).

(c) for polyhedral classes with $S \not= U$, we construct examples by recursive, local manipulations of the pyramids mentioned in (a)  (Subsection~\ref{subsec:polyhedralpair}).

(d) for non-polyhedral classes with  $U > S \geq 6$, we construct examples by recursive, local manipulations starting with polyhedral classes containing simple polyhedra (Subsection~\ref{subsec:nonpolyhedralpair}).

(e) for non-polyhedral classes with $6 \leq U < S$ we provide examples by using the polyhedra
obtained in (d) and the  properties of polarity proved in Section~\ref{sec:prelim}.
We also show how to modify the construction in (d) for this case (Subsection~\ref{subsec:nonpolyhedralpair}).

(f) for monostatic classes with $S=1$ or $U=1$ we provide examples using Conway's polyhedron $P_C$ in class
$(1,4)^E$, we also construct a polyhedron $P_3$ in class $(3,1)$ and subsequently we apply 
recursive, local truncations (Section~\ref{ss:monostatic}).

In Section~\ref{sec:prelim}, we prove a number of  lemmas which help us keep track of the change of the center of mass of a convex polyhedron under local deformations and establish a connection between equilibrium points of a convex polyhedron and its polar.
The local manipulations in our proof may be regarded as generalizations of the algorithm of Steinitz \cite{Grunbaum}.
Figure \ref{fig:proof} summarizes the steps outlined above.

\begin{figure}[ht!]
\begin{center}

\includegraphics[width=\textwidth]{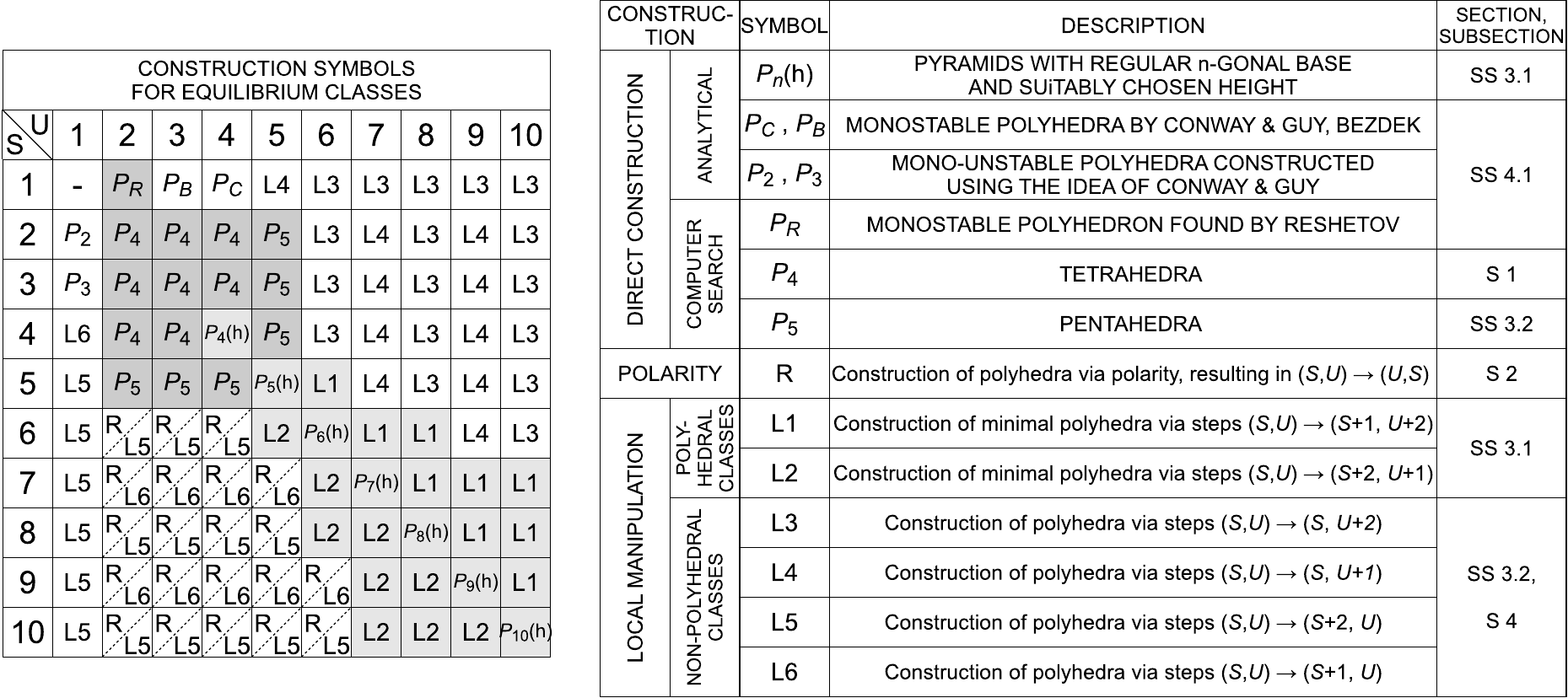}
\caption{Summary of the proof. Left panel: Symbols on the $(S,U)$ grid indicate how polyhedra in the given 
equilibrium class $(S,U)^E$ have been constructed. Dark background corresponds to
classes where polyhedra have been identified by computer search. Light grey background 
corresponds to polyhedral pairs. Symbols are explained in the right panel.
 For $S,U>1$ the indicated constructions provide minimal complexity
and thus the complexity of the class itself. Zero indicates that no polyhedron is known in that class. Right panel: Symbols
in the left panel explained briefly with reference to sections, subsections and sub-subsections of the paper. 
}
\label{fig:proof}
\end{center}
\end{figure}

\begin{figure}[ht!]
\begin{center}
\includegraphics[width=0.8\textwidth]{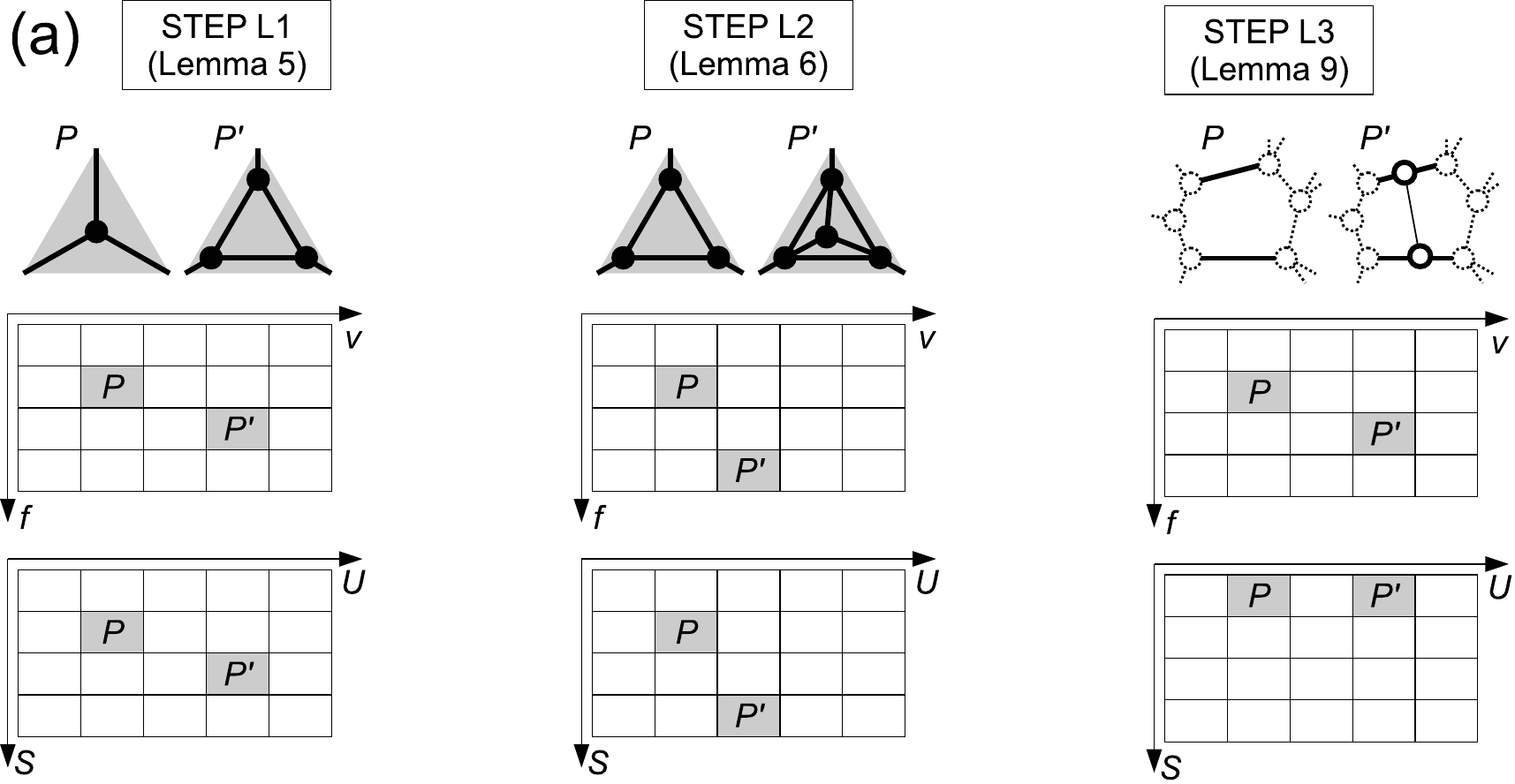}
\vskip 2mm
\includegraphics[width=0.8\textwidth]{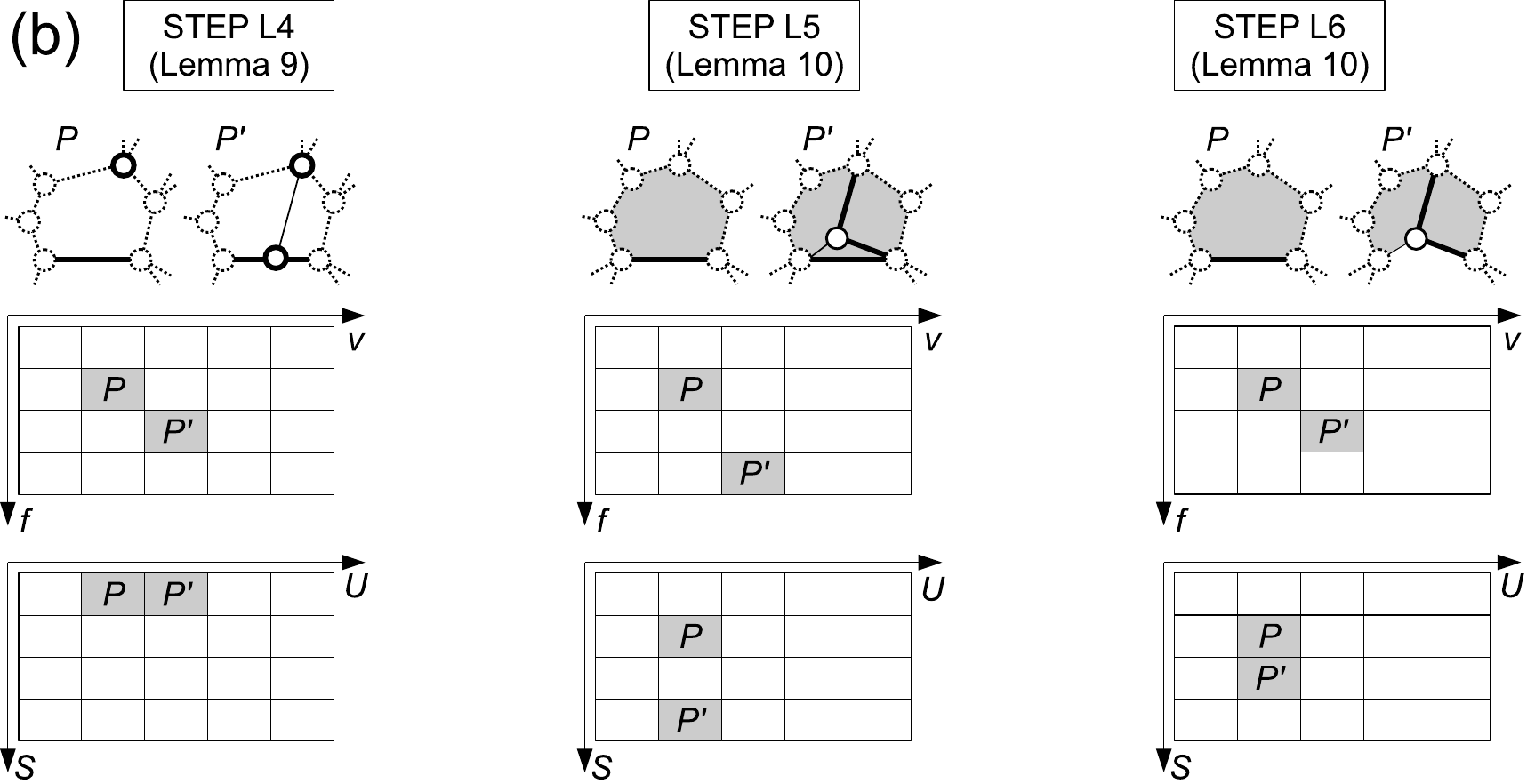}
\caption{Summary of the proof.
(a)-(b) Upper row: schematic picture of local manipulations L1-L6, showing local face structure and equilibria on original and manipulated polyhedra $P$ and $P'$, respectively.  Lower rows: Original and manipulated polyhedra $P$ and $P'$ shown on the $(f,v)$ and $(S,U)$ grids. }
\label{fig:proof1}
\end{center}
\end{figure}

\section{Preliminaries}\label{sec:prelim}

Before we prove some lemmas that we need for Theorem~\ref{thm:main1}, we make a general remark about small truncations:

\begin{rem}\label{rem:smallperturb}
Observe that
\begin{itemize}
 \item[(i)] a nondegenerate (stable) equilibrium point $s_F$ on face $F$ of a convex polyhedron $P$ exists iff the orthogonal projection $s_F$ of $c(P)$ (the center of mass of $P$) onto $F$ is in the relative interior of $F$;
 \item[(ii)] a vertex $q$ is a nondegenerate (unstable) equilibrium point of $P$ iff the plane perpendicular to $q-c(P)$ and containing $q$ contains no other point of $P$;
 \item[(iii)] a nondegenerate equilibrium point $s_E$ on an edge $E$ of $P$ exists iff the orthogonal projection $s_E$ of $c(P)$ onto $E$ is in the relative interior of $E$, and the angle between $c(P)-c_E$ and any of the two faces of $P$ containing $E$ is acute.
\end{itemize}

In the paper, we deal only with a convex polyhedron $P$ which has only nondegenerate equilibria. Then the following observation is used many times in the paper:
\begin{itemize}
\item[(a)] if a vertex $q$ of $P$ is slightly perturbed such that the directions of the edges starting at $q$ change only slightly, then the new vertex is a nondegenerate equilibrium iff $q$ is a nondegenerate equilibrium;
\item[(b)] if an edge $E$ of $P$ is slightly perturbed such that the normal vectors of the two faces containing $E$ change only slightly, then the new edge contains a nondegenerate equilibrium iff $E$ contains a nondegenerate equilibrium;
\item[(c)] if a face $F$ of $P$ is slightly perturbed, then the new face contains a nondegenerate equilibrium iff $F$ contains a nondegenerate equilibrium.
\end{itemize}
It is worth noting that since unstable vertices correspond to local maxima of the Euclidean distance function measured from the center of mass, \emph{any} local perturbation of $P$ yields \emph{at least one} unstable vertex near $q$ in (a). A similar observation can be made for the face $F$ in (c). 
\end{rem}

In the following, $\conv X$, $\aff X$, $\inter X$ and $\cl X$ denote the convex hull, the affine hull, the interior and the closure of the set $X \subset \Re^d$, respectively. The origin is denoted by $o$. For any convex polytope $P$ in $\Re^d$, we denote by $V(P)$ the set of vertices of $P$, and the volume and the center of mass of $P$ by $w(P)$ and $c(P)$, respectively. The polar of the set $X$ is denoted by $X^{\circ}$.

The first three lemmas investigate the behavior of the center of mass of a convex polyhedron under local deformations.

\begin{lem}\label{lem:magnitude}
Let $P$ be a convex polyhedron and let $q$ be a vertex of $P$. Let $P_{\varepsilon}$ be a convex polyhedron such that $P_{\varepsilon} \subset P$, and every point of $P \setminus P_{\varepsilon}$ is contained in the $\varepsilon$-neighborhood of $q$. 
Let $c = c(P)$ and $c_{\varepsilon} = C(P_{\varepsilon})$.
Then there is a constant $\gamma > 0$, independent of $\varepsilon$, such that $|c_{\varepsilon}-c| \leq \gamma \varepsilon^3$ holds for every polyhedron $P_{\varepsilon}$ satisfying the above conditions.
\end{lem}

\begin{proof}
Without loss of generality, let $c=o$, $\bar{c}_{\varepsilon} = c(\cl(P\setminus P_{\varepsilon}))$, $w = w(P)$ and $w_{\varepsilon} = w(P_{\varepsilon})$.
Then $o = w_{\varepsilon} c_{\varepsilon} +  (w-w_{\varepsilon}) \bar{c}_{\varepsilon}$, implying that $c_{\varepsilon} = - \frac{w-w_{\varepsilon}}{w_{\varepsilon}} \bar{c}_{\varepsilon}$.
Note that for some $\gamma' >0$ independent of $\varepsilon$, we have $0 \leq \frac{w-w_{\varepsilon}}{w_{\varepsilon}} < 2\frac{w-w_{\varepsilon}}{w} \leq \gamma' \varepsilon^3$.
Furthermore, for some $\gamma'' > 0$, $|q - \bar{c}_{\varepsilon}| \leq \gamma'' \varepsilon$, which yields that $|\bar{c}_{\varepsilon}|$ is bounded.
Thus, the assertion readily follows.
\end{proof}

\begin{lem}\label{lem:Jacobian_face}
Let $F$ be a triangular face of the convex polyhedron $P$, and assume that each vertex of $P$ lying in $F$ has degree $3$.
Let $q_1$, $q_2$ and $q_3$ be the vertices of $P$ on $F$, and for $i=1,2,3$, let $L_i$ denote the line containing the edge of $P$ through $q_i$ that is not contained in $F$. For $i=1,2,3$ and $\tau \in \Re$, let $q_i(\tau)$ denote the point of $L_i$ at the signed distance $\tau$ from $q_i$, where we orient each $L_i$ in such a way that $q_i(\tau)$ is a point of $P$ for any sufficiently small negative value of $\tau$.
Let $U$ be a neighborhood of $o$, and for any $t=(\tau_1,\tau_2,\tau_3) \in U$, let $W(t) = w(P(t))$ and $C(t)=c(P(t))$, where
$P(t) = \conv \left( (V(P) \setminus \{q_1,q_2,q_3\}) \cup \{ q_1(\tau_1) ,q_2(\tau_2),q_3(\tau_3) \} \right)$.
Then the Jacobian of the function $W(t) C(t)$ is nondegenerate at $t=o$.
\end{lem}

\begin{proof}
It is sufficient to show that the partial derivatives of the examined function span $\Re^3$. Without loss of generality, we may assume that
$q_1$, $q_2$ and $q_3$ are linearly independent.

Consider the polyhedron $P(\tau_1,0,0)$ for some $\tau_1 > 0$, and let $T(\tau_1)= \conv \{ q_1,q_2,q_3,q_1(\tau_1) \}$,
$\bar{W}(\tau_1) =w(T(\tau_1))$ and $\bar{C}(\tau_1) = c(T(\tau_1))$.
Let $A$ be the area of the triangle $\conv \{ q_1,q_2,q_3\}$.
If $\tau_1 > 0$ is sufficiently small, then
\[
\left. \frac{\partial}{\partial \tau_1} W(t)C(t) \right|_{t=(0,0,0)} = \frac{\sin \alpha_1 A}{12}(2q_1+q_2+q_3),
\]\[
W(\tau_1,0,0) C(\tau_1,0,0) = w(P) c(P)+\bar{W}(\tau_1) \bar{C}(\tau_1).
\]
Since $\bar{C}(\tau_1)=\frac{1}{4}(q_1+q_2+q_3+q_1(\tau_1))$, it follows that
\[
\left. \frac{\partial}{\partial \tau_1} W(t)C(t) \right|_{t=(0,0,0)} = \frac{\sin \alpha_1 A}{12}(2q_1+q_2+q_3),
\]
where $\alpha_i$ denotes the angle between $L_i$ and the plane through $q_1,q_2,q_3$.

Using a similar consideration, we obtain the same formula if $\tau_1 < 0$, and similar formulas, where $q_2$ or $q_3$ plays the role of $q_1$, in the partial derivatives with respect to $\tau_2$ or $\tau_3$, respectively.
Note that $0 < \alpha_1, \alpha_2, \alpha_3 \leq \frac{\pi}{2}$. Thus, to show that the three partial derivatives are linearly independent, it suffices to show that
the vectors $2q_1+q_2+q_3$, $q_1+2q_2+q_3$ and $q_1+q_2+2q_3$ are linearly independent. To show it under the assumption that $q_1,q_2,q_3$ are linearly independent can be done using elementary computations, which we leave to the reader.
\end{proof}

\begin{rem}\label{rem:Jacobian_vertex}
We remark that Lemma~\ref{lem:Jacobian_face} can be `dualized' in the following form: Assume that $q$ is a $3$-valent vertex of $P$, and each face of $P$ that $q$ lies on is a triangle. Furthermore, let $U$ be a neighborhood of $q$, and for any $x \in U$, let $W(x) = w \left( \conv \left( (V(P) \setminus \{ q \}) \cup \{ x \} \right) \right)$, and $C(x) = c \left( \conv \left( (V(P) \setminus \{ q \}) \cup \{ x \} \right) \right)$.
Then the Jacobian matrix of the function $W(\cdot) C(\cdot) : U \to \Re^3$ is nondegenerate at $q$.
\end{rem}

\begin{rem}\label{rem:Jacobian}
If the Jacobian of a smooth vector-valued function in $\Re^3$ is nondegenerate, by the Inverse Function Theorem it follows that the function is surjective. Thus, a geometric interpretation of Lemma~\ref{lem:Jacobian_face} and Remark~\ref{rem:Jacobian_vertex} is that under the given conditions, by slight modifications of a vertex or a face of $P$ the function $w(P) c(P)$ moves everywhere within a small neighborhood of its original position. 
\end{rem}

In the forthcoming two lemmas we investigate the connection between polarity and equilibrium points.

\begin{lem}\label{lem:simplices}
Let $S$ be a nondegenerate simplex in the Euclidean space $\Re^d$ such that $o \in \inter S$.
Then $o= c(S^{\circ})$ if, and only if $o = c(S)$.
\end{lem}

\begin{proof}
Let the vertices of $S$ be denoted by $p_1, p_2, \ldots, p_{d+1}$. For $i=1,2,\ldots,d+1$, let $n_i$ denote the orthogonal projection of $o$ onto the facet hyperplane $H_i$ of $S$ not containing $p_i$, and let $H'_i$ be the hyperplane through $o$ and parallel to $H_i$. We remark that since $o \in \inter S$, none of the $p_i$s and the $n_i$s is zero. Finally, let $\alpha_i$ denote the angle between $p_i$ and $n_i$.

Assume that $o = c(S)$. Then for all values of $i$, we have $\dist(p_i,H'_i)= d \dist(H'_i,H_i)$, where $\dist(A,B) = \inf \{ |a-b|: a \in A, b \in B \}$ is the distance of the sets $A$ and $B$. This implies that the projection of $p_i$ onto the line through $o$ and $n_i$ is $-dn_i$ for all values of $i$, or in other words,
\begin{equation}\label{eq:polarity}
\cos \alpha_i |p_i| = -d|n_i|
\end{equation}
for all values of $i$. On the other hand, it is easy to see that if (\ref{eq:polarity}) holds for all values of $i$, then $o=c(S)$.

The vertices of $S^{\circ}$ are the points $p^{\star}_i=\frac{n_i}{|n_i|^2}$, where $i=1,2,\ldots,d+1$, and the projection of $o$ onto the facet hyperplane of $P^{\circ}$ not containing $p^{\star}_i$ is $n^{\star}_i=\frac{p_i}{|p_i|^2}$. Hence, the angle between $p^{\star}_i$ and $n^{\star}_i$ is $\alpha_i$.
Similarly like in the previous paragraph, $o=c(S^{\circ})$ if, and only if
\begin{equation}\label{eq:polarity2}
\cos \alpha_i |p_i^{\star}| = -d|n_i^{\star}|
\end{equation}
holds for all values of $i$.
On the other hand, if $\cos \alpha_i |p_i| = -d|n_i|$ for some value of $i$, then $\cos \alpha_i |p_i^{\star}| = \frac{\cos \alpha_i}{|n_i|} = - \frac{d}{|p_i|} = -d|n_i^{\star}|$, and vice versa. Thus, (\ref{eq:polarity}) and (\ref{eq:polarity2}) are equivalent, implying Lemma~\ref{lem:simplices}.
\end{proof}

\begin{lem}\label{lem:polarity}
Let $P$ be a convex $d$-polytope in the Euclidean space $\Re^d$ such that $o \in \inter P$, and let $P^{\circ}$ be its polar.
Let $F$ be a $k$-face of $P$, where $0 \leq k \leq d-1$,  and let $F^{\star}$ denote the corresponding $(d-k-1)$-face of $P^{\circ}$.
Then $F$ contains a nondegenerate equilibrium point of $P$ with respect to $o$ if, and only if $F^{\star}$ contains a nondegenerate equilibrium point
of $P^{\circ}$ with respect to $o$.
\end{lem}

\begin{proof}
Let $F = \conv \{ p_i : i \in I \}$, where $I$ is the set of the indices of $P$ such that $p_i$ is contained in $F$, and let $p$ be the orthogonal projection of $o$ onto $\aff F$. Let $L=\aff (F \cup \{ o \})$,
and let $L^c$ denote the orthogonal complement of $L$ passing through $o$.
For any facet hyperplane of $P$ containing $F$, let $n_j$, $j \in J$ denote the projection of $o$ onto this hyperplane. 
Let $H_j^+$ be the closed half space $\{ q \in \Re^d : \langle q,n_j \rangle \leq \langle n_j , n_j \rangle \}$.
Let $\bar{H}_i^+= H_i^+ \cap H$ for any $i \notin I$. Finally, let $\bar{n}_i$ be the component of $n_i$ parallel to $H$.

Before proving the lemma, we observe that for any given vectors $n_1,n_2,\ldots, n_k$ spanning $\Re^d$, the following are equivalent:
\begin{itemize}
\item[(a)] $o$ is an interior point of a polytope $Q$ in $\Re^d$ with outer facet normals $n_1, n_2, \ldots, n_k$.
\item[(b)] There are some $\lambda_1, \lambda_2,\ldots, \lambda_k > 0$ such that $o \in \inter Q'$, where $Q'=\conv \{ \lambda_1 n_1, \lambda_2 n_2, \ldots, \lambda_k n_k \}$.
\item[(c)] We have $o \in \inter \conv \{ \lambda_1 n_1, \lambda_2 n_2, \ldots, \lambda_k n_k \}$ for any $\lambda_1, \lambda_2, \ldots, \lambda_k > 0$.
\end{itemize}
We note that if a polytope $Q$ satisfies the conditions in (a), then its polar $Q'=Q^{\circ}$ satisfies the conditions in (b), and vice versa.
Finally, observe that if $F$ contains an equilibrium point, then by exclusion it is $p$.

We show that $p$ is a nondegenerate equilibrium point of $F$ if, and only if it is contained in the relative interiors of the conic hulls of the $p_i$s as well as those of the $n_j$s.
First, let $p$ be a nondegenerate equilibrium point. Then $p \in \relint F$, that is, it is in the relative interior of the conic hull (in particular, the convex hull) of the $p_i$s.
Observe that since the projection of $o$ onto $\aff F$ is $p$, for any $j \in J$, the projection of $n_j$ onto $\aff F$ is $p$.
In other words, $n_j \in L' =\aff(L^c \cup \{ p \})$ for all $j \in J$. Since $p$ is a vertex of the polytope $P \cap L'$,
the vectors $n_j$, $j \in J$ span this linear subspace, or equivalently, the vectors $\bar{n}_j$ span $L^c$.
Observe that the intersection of $P$ with the affine subspace $(1-\varepsilon) p + L^c$, for sufficiently small values of $\varepsilon > 0$, is a $(d-k-1)$-polytope, with outer facet normals $\bar{n}_j$, $j \in J$, which contains $(1-\varepsilon) p$ in its relative interior. By the observation in the previous paragraph, it follows that $o$ is contained in the relative interior of the convex hull of the $\bar{n}_j$s, which implies that $p$ is contained in the relative interior of the conic hull of the $n_j$s.
On the other hand, if $p$ is contained in the relative interior of the conic hull of the $p_i$s, then the fact that $p \in \aff F$ implies that $p \in \relint F$.
Furthermore, if $p$ is contained in the relative interior of the conic hull of the $n_j$s, then $o$ is contained in the relative interior of the convex hull of the $\bar{n}_j$s. Thus, the only solution for $q \in L^c$ of the system of linear inequalities $\langle q, \bar{n}_j \rangle \leq 0$, where $j \in J$, is $q=p$, which implies that the only point of $P$ in $p+L^c$ is $p$. This means that $p$ is a nondegenerate equilibrium point of $P$.

Finally, observe that the vertices of $F^{\star}$ are the points $\frac{n_j}{|n_j|^2}$, and the projections of $o$ onto the facet hyperplanes of $P^{\circ}$ containing $F^{\star}$ are the points $\frac{p_i}{|p_i|^2}$. Furthermore, $\aff F^{\star} = \frac{p}{|p|^2} + L^c$, which yields that the projection of $o$ onto $\aff F^{\star}$ is $\frac{p}{|p|^2}$. Combining it with the consideration in the previous paragraph, this yields the assertion.
\end{proof}

The next corollary is an immediate consequence of Lemmas~\ref{lem:simplices} and \ref{lem:polarity} and, together with the result of Conway \cite{Dawson}, implies Theorem~\ref{thm:tetra}.

\begin{cor}\label{cor:antistable}
Every homogeneous tetrahedron has at least two vertices which are equilibrium points. Furthermore, there are inhomogeneous tetrahedra with exactly one vertex which is an equilibrium point.
\end{cor}

\section{Polyhedra with many stable or unstable equilibria: proof of Theorem~\ref{thm:main1}}\label{sec:proof}

\subsection{Proof of Theorem~\ref{thm:main1} for polyhedral pairs}\label{subsec:polyhedralpair}

We need to show that if the class  $(S,U)^E$ is defined by a polyhedral pair, then there is a polyhedron with $S$ faces and $U$ vertices. For brevity, we call such a polyhedron a \emph{minimal polyhedron} in class $(S,U)^E$. We dothe construction separately in several cases.

\subsubsection{Case 1} $S=U \geq 4$.\\
Let $S \geq 4$, and consider a regular $(S-1)$-gon $R_S$ in the $(x,y)$-plane, centered at $o$ and with unit inradius. Let $P_v(h)$ be the pyramid with base $R_v$ and apex $(0,0,h)$. By its symmetry properties, $P_S(h)$ is a minimal polyhedron in the class  $( S,S)^E$ for all $h > 0$.

\subsubsection{Case 2} $S > 4$ and $S < U \leq 2S-4$.\\
In this case the proof is based on Lemma~\ref{lem:inductive_step_12}.

\begin{lem}\label{lem:inductive_step_12}
Assume that $P$ is a minimal polyhedron in class  $(S,U)^E$ having a vertex of degree $3$. Then there is a minimal polyhedron in class  $(S+1, U+2)^E$ having a vertex of degree $3$.
\end{lem}

\begin{proof}
Let $P$ be a minimal polyhedron in class  $(S, U)^E$ with a vertex $q$ of degree $3$.
For sufficiently small $\varepsilon > 0$, let $P_{\varepsilon} \subset P$ be the intersection of $P$ with the closed half space with inner normal vector $c-q$, at the distance $\varepsilon$ from $q$. We show that if $\varepsilon$ is sufficiently small, then $P_{\varepsilon}$ satisfies the conditions in the lemma.

If $\varepsilon$ is sufficiently small, the boundary of this half space intersects only those edges of $P$ that start at $q$.
Thus, $P_{\varepsilon}$ has one new triangular face $F$, and three new vertices $q_1, q_2, q_3$ on $F$. Since $q$ is not a vertex of $P_{\varepsilon}$, $P_{\varepsilon}$ has $S+1$ faces and $U+2$ vertices. Furthermore, $q_1, q_2$ and $q_3$ have degree $3$, which means that we need only to show that $P_{\varepsilon}$ is a minimal polyhedron.
To do it, we set $c=c(P)$ and $c_{\varepsilon} = c(P_{\varepsilon})$.

Note that by (\ref{Euler}) and (\ref{Poincare}), every edge of a minimal polyhedron contains an equilibrium point.
Thus, by Remark~\ref{rem:smallperturb}, if $\varepsilon$ is sufficiently small, then every edge of $P_{\varepsilon}$, apart from those in $F$, contains an equilibrium point with respect to $c_{\varepsilon}$.
We intend to show that if $\varepsilon$ is sufficiently small, then the edges of $P_{\varepsilon}$ in $F$ also contain equilibrium points with respect to $c_{\varepsilon}$, which, by (\ref{Euler}) and (\ref{Poincare}) clearly implies that $P_{\varepsilon}$ is a minimal polyhedron.

Consider, e.g. the edge $E=[q_1,q_2]$, and let $F_3$ be the face of $P_{\varepsilon}$ different from $F$ and containing $E$.
Let $h$ and $s$ be the equilibrium point on $E$ and on $F_3$, respectively, with respect to $c$.
Let $\alpha$ and $\beta$ denote the dihedral angles between the planes $\aff (E \cup \{ c\})$ and $\aff F$, and the planes $\aff (E \cup \{ c\})$ and $\aff F_3$, respectively. The fact that $h$ is an equilibrium point with respect to $c$ is equivalent to saying that the orthogonal projection of $c$ onto the line of $E$ is $h$, and that $0 < \alpha, \beta < \frac{\pi}{2}$.

Since $h$ is contained in the plane $\aff \{ q,c,s\}$ for all values of $\varepsilon$, and it is easy to see that there is some constant $\gamma' > 0$ independent of $\varepsilon$ such that $|q_1-h|, |q_2 - h| \geq \gamma' \varepsilon$.
Similarly, an elementary computation shows that for some constant $\gamma'' > 0$ independent of $\varepsilon$, we have $0 < \alpha, \beta \leq \frac{\pi}{2} - \gamma'' \varepsilon$. Thus, Lemma~\ref{lem:magnitude} implies that for small values of $\varepsilon$, $E$ contains an equilibrium point with respect to $c_{\varepsilon}$, implying that $P_{\varepsilon}$ is a minimal polyhedron.
\end{proof}

Now, consider some class  $(S,U)^E$ with $S > 4$ and $S < U \leq 2S-4$. Then, if we set $k=U-S$ and $S_0=S-k$, we have $0 < k \leq S-4$
and $4 \leq S_0$. In other words,  $(S,U)^E = (S_0+k,S_0+2k )^E$ for some $S_0 \geq 4$ and $k > 0$.
Now, by the proof in Case 1, the class  $(S_0 , S_0)^E$ contains a minimal polyhedron, e.g. a right pyramid $P_{S_0}(h)$ with a regular $(S_0-1)$-gon as its base, where $h >0$ is arbitrary. Note that the degree of every vertex of $P_{S_0}(h)$ on its base is $3$, and thus, applying Lemma~\ref{lem:inductive_step_12} yields a minimal polyhedron in class  $(S_0+1,S_0+2)^E$ having a vertex of degree $3$. Repeating this argument $(k-1)$ times, we obtain a minimal polyhedron in class  $(S,U)^E$.

\subsubsection{Case 3} $S > 4$ and $\frac{S}{2}+2 \leq U < S$.\\
Note that these inequalities are equivalent to $U > 4$ and $U < S \leq 2U-4$.
For the proof in this case we need Lemma~\ref{lem:inductive_step_21}.

\begin{lem}\label{lem:inductive_step_21}
Assume that there is a minimal polyhedron $P$ in class  $(S,U)^E$ having a triangular face. Then there is a minimal polyhedron $P'$ in class $(S+2,U+1)^E$ having a triangular face $F'$.
\end{lem}

\begin{proof}
Let $c=c(P)$, and let $c_F$ be its orthogonal projection on the plane of $F$. Since $P$ is a minimal polyhedron, $c_F$ is a relative interior point of $F$, and an equilibrium point with respect to $c$ (see also Fig.~\ref{fig:inductive_step_21} for illustration).
Let $\bar c$ be the centroid of $F$ and define the vector $u$ as $\bar c-c_F$.
Let $v$ be the outer unit normal vector of $F$, and for any $0 < \varepsilon$ and $0 \leq \alpha\leq 1$, let $T_{\varepsilon\alpha}$ denote the tetrahedron with base $F$ and apex $q = c_F + \varepsilon v + \alpha u$ such that $T_{\varepsilon\alpha} \cap P = F$.
Let $P_{\varepsilon\alpha} = T_{\varepsilon\alpha} \cup P$, $c'=c(P_{\varepsilon\alpha})$, and $c'_F$ be the orthogonal projection of $c'$ on the plane of $F$.
By Remark~\ref{rem:smallperturb}, for a sufficiently small $\varepsilon$, equilibrium points on all vertices of $P_{\varepsilon\alpha}$ except $q$, as well as on all edges and faces of $P_{\varepsilon\alpha}$ not containing $q$ will be preserved.

\begin{figure}[ht]
\begin{center}
\includegraphics[width=0.7\textwidth]{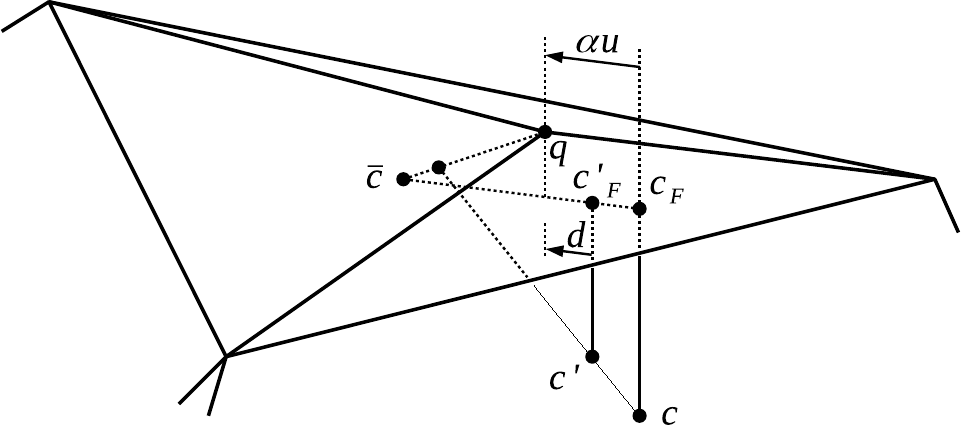}
\caption{Building a tetrahedron on a triangular face of a convex polyhedron}
\label{fig:inductive_step_21}
\end{center}
\end{figure}

It is also easy to see from simple geometric considerations that for small values of $\varepsilon$, every face and vertex of $P_{\varepsilon\alpha}$ contains an equilibrium point with respect to $c'$ if $q$, $c'_F$ and $c'$ are collinear.
In the special case of $u=0$, those points are obviously collinear. In any other case, it is also straightforward to see that $c'_F \in \relint \conv \{c_F,c_F/4+3\bar c /4 \}$. Let us define $d(\alpha) = (c_F + \alpha u - c'_F)\cdot u$. Since $d$ continuously varies with $\alpha$ and $d(0)<0, d(1)>0$, for any small $\varepsilon>0$ there is an $\alpha_0$ such that apex  $q$ and all edges and faces it is contained in have equilibrium points.
\end{proof}

\begin{rem}\label{rem:21_later}
Note that the argument also yields a polyhedron $P'$ such that there is an equilibrium point on each face and at every vertex of $P'$ \emph{with respect to the original reference point}: in this case we may choose the value of $\alpha$ in the proof simply as $\alpha=0$.
\end{rem}

Now we prove Theorem~\ref{thm:main1} for Case 3.
Like in Case 2, if we set $k=S-U$ and $S_0=U-k$, then $S_0 \geq 4$, $k > 0$, and  $(S,U)^E = ( S_0 +2k, S_0+k)^E$.
Consider the right pyramid $P_{S_0}(h)$ in Case 1. This pyramid has $S_0$ faces consisting of $S_0-1$ triangles and one regular $(S_0-1)$-gon shaped face.
Thus, applying the construction in Lemma~\ref{lem:inductive_step_21} $k$ times subsequently yields the desired polyhedron.

\subsection{Proof of Theorem~\ref{thm:main1} for non-polyhedral pairs}\label{subsec:nonpolyhedralpair}
\subsubsection{Case 1}  $2 \leq S \leq 4$ and $2 \leq U \leq 4$.

\begin{lem}\label{lem:complexity_tetrahedron}
Let $S,U \in \{2,3,4\}$. Then $C(S,U)=2R(S,U)$.
\end{lem}

\begin{proof}
Table \ref{9classes} contains an example for a tetrahedron in each of the 9 classes 
(illustrated in Figure \ref{fig:tetra}) and for the
tetrahedron we have $n=f+v+e=14$, consequently an upper bound for complexity can be computed as
$C(S,U)\leq 14-S-U-H=16-2S-2U$. Since  from  (\ref{Rformula}) we have the same for the lower bound
we proved the claim.
\end{proof}

\begin{table} [ht] 
{\tiny
\begin{tabular}{|c||c|c|c|c|c|c|c|c|c|c|c|c|c|c|c|c|c|c|c|}
\hline
 & \multicolumn{5}{c} {Non-constant vertex coordinates}  \vline & \multicolumn{14}{c} {Equilibria on}  \vline  \\
\cline{7-20}
 & \multicolumn{5}{c} {                               }  \vline & \multicolumn{4}{c} {faces}   \vline & \multicolumn{4}{c} {vertices}  \vline & \multicolumn{6}{c} {edges}  \vline \\
\cline{2-20}
Class&  $C_x$ &$C_y$ & $D_x$ & $D_y$ & $D_z$ & {\rotatebox[origin=c]{90}{$ABC$}} & {\rotatebox[origin=c]{90}{$ABD$}} & {\rotatebox[origin=c]{90}{$ACD$}} & {\rotatebox[origin=c]{90}{$BCD$}}
& $A$ & $B$ & $C$ & $D$ & {\rotatebox[origin=c]{90}{$AB$}} & {\rotatebox[origin=c]{90}{$AC$}} & {\rotatebox[origin=c]{90}{$AD$}} & {\rotatebox[origin=c]{90}{$BC$}} & {\rotatebox[origin=c]{90}{$BD$}} & {\rotatebox[origin=c]{90}{$CD$}} \\
\hline
\hline
$(2,2)$   & 3.2 & 1.9 & -2.2 & 0.3 & 1.8 & 0 & 0 & 1 & 1 & 0 & 0 & 1 & 1 & 1 & 0 & 0 & 0 & 0 & 1 \\
\hline
$(2,3)$  & 1.9 & 5.3 & 1.9 & -0.9 & 5.2 & 0 & 0 & 1 & 1 & 1 & 0 & 1 & 1 & 0 & 1 & 1 & 0 & 0 & 1 \\
\hline
$(2,4)$  &  -0.9 & 5.3 & 1.9 & 0.9 & 5.2 & 0 & 0 & 1 & 1 & 1 & 1 & 1 & 1 & 1 & 1 & 0 & 0 & 1 & 1 \\
\hline
$(3,2)$  & 1.0 & 2.7 & -0.9 & -4.1 & 3.4 & 0 & 1 & 1 & 1 & 0 & 0 & 1 & 1 & 0 & 0 & 1 & 0 & 1 & 1 \\
\hline
$(3,3)$  & 1.0 & 5.7 & 0.5 & -0.5 & 1.3 & 1 & 0 & 1 & 1 & 1 & 0 & 1 & 1 & 0 & 1 & 1 & 1 & 0 & 1 \\
\hline
$(3,4)$  & 0.5 & 2.8 & 0.5 & -0.7 & 1.2 & 1 & 0 & 1 & 1 & 1 & 1 & 1 & 1 & 0 & 1 & 1 & 1 & 1 & 1 \\
\hline
$(4,2)$  & 3.2 & 3.8 & -2.2 & -2.9 & 2.5 & 1 & 1 & 1 & 1 & 0 & 0 & 1 & 1 & 1 & 1 & 0 & 0 & 1 & 1 \\
\hline
$(4,3)$  & 1.9 & 5.3 & 1.9 & 5.0 & 1.8 & 1 & 1 & 1 & 1 & 1 & 0 & 1 & 1 & 0 & 1 & 1 & 1 & 1 & 1 \\
\hline
\end{tabular}
}
\normalsize
\vspace{0.5cm}
\caption{Examples for tetrahedra in equilibrium classes $(S,U)^E$, $S,U \in \{2,3,4\}$, $(S,U) \not=4,4$. Constant vertex coordinates for all tetrahedra are
$A_x=A_y=A_z=B_y=C_z=0$, $B_x=1$.}
\label{9classes}
\end{table}

\begin{figure}[ht]
\begin{center}
\includegraphics[width=\textwidth]{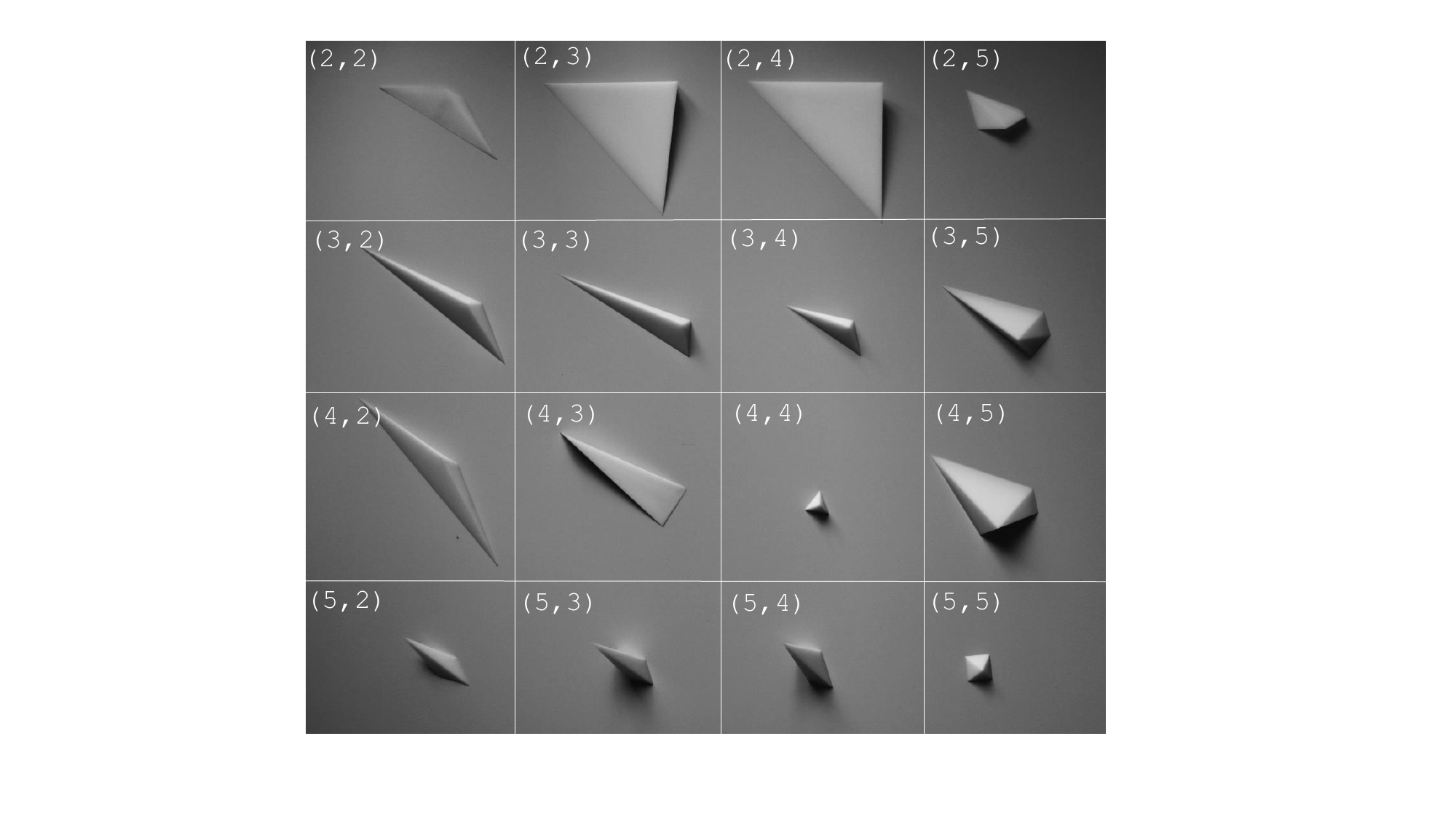}
\caption{The 8 tetrahedra in Table \ref{9classes} and the 6 pentahedra in Table \ref{6classes}, the regular tetrahedron and the symmetrical pyramid in equilibrium classes $(S,U)^E$, $S,U \in \{2,3,4,5\}$
produced by 3D printing.}
\label{fig:tetra}
\end{center}
\end{figure}

\subsubsection{Case 2}  $2 \leq S \leq 4, U=5$ or $2 \leq U \leq 4, S=5$ .\\
This case follows from Lemma~\ref{lem:complexity_pentahedron}.

\begin{lem}\label{lem:complexity_pentahedron}
Let $2 \leq S \leq 4, U=5$ or $2 \leq U \leq 4, S=5$. Then $C(S,U)=2R(S,U)$.
\end{lem}

\begin{proof}
Table \ref{6classes} contains an example for a  pentahedron in each of the 6 classes
(illustrated in Figure \ref{fig:tetra}) and for the
pentahedron we have $n=f+v+e=18$, consequently an upper bound for complexity can be computed as
$C(S,U)\leq 18-S-U-H=20-2S-2U$. From (\ref{Rformula}) we obtain the same lower bound for all 6 classes
so we proved the claim.
\end{proof}

\begin{table} [ht] 
{\tiny
\begin{tabular}{|c||c|c|c|c|c|c|c|}
\hline
 Class & \multicolumn{7}{c} {Non-constant vertex coordinates} \vline  \\
\hline
&  $C_x$ &$C_y$ & $D_x$ & $D_y$ & $E_x$ & $E_y$ & $E_z$ \\
\hline
\hline
$(2,5)$   & 1.0 & 1.7 & 0.5 & -0.3 & 2.1 & 1.2 & 1.2  \\
\hline
$(3,5)$    & 1.0 & 1.7 & 3.8 & -2.2 & 1.6 & 0.9 & 0.9  \\
\hline
$(4,5)$   & 2.5 & 1.4 & 3.8 & -2.2 & 2.0 & 1.2 & 1.2  \\
\hline
$(5,2)$   & 1.0 & 1.7 & 0.9 & 0.5 & -0.6 & -1.1 & -1.1  \\
\hline
$(5,3)$   & 1.0 & 1.7 & 0.9 & 0.5 & 1.5 & 2.6 & 2.6  \\
\hline
$(5,4)$   & 1.0 & 1.7 & 1.3 & 0.8 & 1.5 & 2.6 & 2.6  \\
\hline
\end{tabular}
}
\normalsize
\vspace{0.5cm}
\caption{ Examples for pentahedra in equilibrium classes $(i,5)$ and $(5,i)$ $i \in \{2,3,4\}$. Constant vertex coordinates for all pentahedra are
$A_x=A_y=A_z=B_x=C_z=D_z=0$, $B_y=1$.}
\label{6classes}
\end{table}
\subsubsection{Case 3} $S \geq 5$ and $U > 2S-4$, or $2 \leq S \leq 4$ and $U \geq 6$.
First, we prove the following lemma.

\begin{lem}\label{lem:face_trunc}
Let $P \in (S,U)^E$ be a convex polyhedron with $f$ faces and $v$ vertices. Let $q_i , i=1...j$ be successive vertices of an $m$-gonal ($m\geq, j\geq 3$) face $F$ of $P$ such that 
\begin{itemize}
 \item[i)] the lines $\aff(\{q_1, q_2\})$ and $\aff(\{q_{j-1}, q_j\})$ intersect at some point $q$ with the property $\left| q-q_1 \right| > \left| q-q_2 \right|$; 
 \item[ii)] both edges $E_a = [q_1, q_2]$ and $E_b= [q_{j-1}, q_j]$ contain saddle points; 
 \item[iii)] the vertices $q_i , i = 2...j-1$ are trivalent.
\end{itemize}
Then there is convex polyhedron $P' \in (S,U+2)^E$ with $f+1$ faces and $v+2$ vertices.
\end{lem}
\begin{figure}[ht]
\begin{center}
\includegraphics[width=0.9\textwidth]{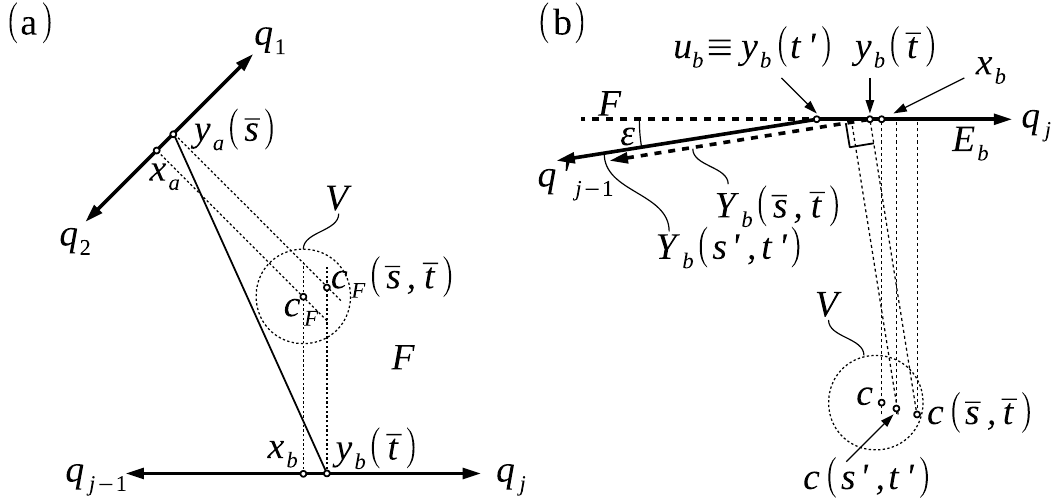}
\caption{Increasing the number of unstable equilibria by two. Views perpendicular to the plane $F$ (a) and edge $[u_a,u_b]$ (b).}
\label{fig:uplus2}
\end{center}
\end{figure}
\begin{proof}
Let the saddle points on $E_a$ and $E_b$ be denoted by $x_a$ and $x_b$.
In the proof, based on Remark~\ref{rem:smallperturb}, we show that there is an arbitrarily small truncation of $P$ by a plane that intersects $F$ in a line close to $x_a$ and $x_b$ that results in two new unstable vertices $u_a$ and $u_b$.

We choose a suitable truncation from a $2$-parameter family of truncations defined as follows: For any $t \in [0,1]$, set $y_a(t) = t q_2 + (1-t)q_1$ and $y_b(t) = t q_{j-1}+ (1-t) q_j$. Let $G(s,t)$ be the plane that intersects $[q_1,q_2]$ at $y_a(s)$ and $[q_{j-1},q_j]$ at $y_b(t)$, whose angle with the plane of $F$ is a sufficiently small value $\varepsilon > 0$ (the term `sufficiently small' is explained in the next paragraph) and truncates the vertices $q_2,q_3, \ldots, q_{j-1}$. For $i=2,3,\ldots,j-1$, let $q_i(s,t)$ be the intersection of $G(s,t)$ with the edge of $P$ starting at $q_i$ and not contained in $F$. Finally, let $P(s,t)$ be the truncation of $P$ by $G(s,t)$, that is, $P(s,t) = \cl(P \setminus \conv \{ y_a(s), y_b(t), q_2,...,q_{j-1}, q_2(s,t),\ldots,q_{j-1}(s,t) \}$.
We denote the center of mass of $P(s,t)$ by $c(s,t)$, and the projection of $c$ and $c(s,t)$ onto the plane of $F$ by $c_F$ and $c_F(s,t)$, respectively.
Furthermore, we denote the new edge of $P(s,t)$ starting at $y_a(s)$ and different from $[y_a(s),y_b(t)]$ by $Y_a(s,t)$, and define $Y_b(s,t)$ similarly.

We choose some $\varepsilon > 0$ to satisfy the following conditions: with respect to any point $c' \in V$, the \emph{original polyhedron} $P$ has equilibrium points on the same faces and edges, and at the same vertices, as with respect to the center of mass $c$ of $P$, where $V$ is the locus of the centers of mass of all truncations of $P$ by the plane $G(s,t)$, $s,t \in [0,1]$. Furthermore, we assume also that $G(s,t)$ truncates no vertex or equilibrium point of $P$ other than those on $F$, and that there is some arbitrarily small, fixed value $\delta > 0$ (independent of $(s,t)$) such that $c(s,t)$ is a Lipschitz function at every $(s,t)$ with Lipschitz constant $\delta$, i.e. $| c(s+\Delta s, t + \Delta t) - c(s,t)| \leq \delta \sqrt{(\Delta s)^2+(\Delta t)^2}$ for all $s,t \in [0,1]$.

First, we show that for some suitable choice of $s$ and $t$, the orthogonal projections of $c(s,t)$ onto the lines of $E_a$ and $E_b$ are 
$y_a(s)$ and $y_b(t)$, respectively. To do this, we use a consequence of Brouwer's fixed point theorem, the so-called Cube Separation Theorem from \cite{Postnikov}, which states the following:
Let the pairs of opposite facets of a $d$-dimensional cube $K$ be denoted by $F'_i$ and $F''_i$, $i=1,2, \ldots, d$, and let $C_i$, $i=1,2,\ldots,d$ be compact sets such that $C_i$ `separates' $F'_i$ and $F''_i$, or in other words, $K \setminus C_i$ is the disjoint union of two open sets $Q'_i, Q''_i$ such that $F_i' \subset Q'_i$, and $F''_i \subset Q''_i$. Then $\bigcap_{i=1}^d C_i \neq \emptyset$.

To apply this theorem, we set $K= \{ (s,t) : 0 \leq s,t \leq 1 \}$, and define $Q'_1$, $C_1$ and $Q''_1$ as the set of pairs $(s,t)$ such that the orthogonal projection of $c(s,t)$ onto the line of $E_a$ is a relative interior point of $[y_a(s), q_1]$, coincides with $y_a(s)$, or does not belong to $[y_a(s), q_1]$, respectively. We define $Q'_2$, $C_2$ and $Q''_2$ similarly. Then these sets satisfy the conditions of theorem, and we obtain a pair $(\bar{s},\bar{t})$ with the desired property.
Note that by the choice of $\varepsilon > 0$, it holds that in a neighborhood of $(\bar{s},\bar{t})$, the orthogonal projection of $c(s,t)$ onto the line of $Y_a(s,t)$ is in the relative interior of $Y_a(s,t)$, and the same holds also for the projection onto the line of $Y_b(s,t)$.
Now we choose some $(s',t')$ sufficiently close to $(\bar{s},\bar{t})$ such that the intersections of $G(s',t')$ and $G(\bar{s},\bar{t})$ with $F$ are parallel, and that of $G(s',t')$ is closer to $q_2$ and $q_{j-1}$ than that of $G(\bar{s},\bar{t})$.
By the Lipschitz property of $c(s,t)$, we have that the distance of the two intersection lines is greater than $|c(s',t') - c(\bar{s},\bar{t})|$, and hence, the projections of $c(s',t')$ onto the lines of $E_a$ and $E_b$ lie in the relative interior of the segments $[y_a(s'),q_1]$ and $[y_b(t'),q_j]$, respectively. From this it readily follows that both these edges of $P'=P(s',t')$ and also $Y_a(s',t')$ and $Y_b(s',t')$ contain saddle points with respect to $c(s',t')$.
This implies also that $y_a(s')$ and $y_b(t')$ are vertices of $P'$ carrying unstable equilibrium points, and the assertion follows.
\end{proof}

\begin{cor}\label{cor:face_trunc}
Let conditions (i) and (iii) of Lemma~\ref{lem:face_trunc} hold and (ii) be modified as follows:

ii) $q_1$ contains an unstable and $E_b=[ q_{j-1}, q_j]$ contains a saddle-type equilibrium point.

Then there exists a polyhedron $P'' \in (S,U+1)^E$ with $f+1$ faces and $v+1$ vertices.
\end{cor}

\begin{rem}\label{rem:face_trunc}
A simplified version of the proof of Lemma~\ref{lem:face_trunc} can be used to prove the same statement for a \emph{fixed} reference point $c$.
\end{rem}

To prove Theorem~\ref{thm:main1} in Case 3, we construct a simple polyhedron with $U$ vertices that has $S$ stable and $U$ unstable points. Since any polyhedron in class $(S, U)^E$ has at least $U$ vertices, and among polyhedra with $U$ vertices those with a minimum number of faces are the simple ones, such a polyhedron clearly has minimal mechanical complexity in class $(S,U)^E$.

First, consider the case that $S \geq 5$ and $U > 2S-4$.                      
Let $U_0 = 2S-4$. By the construction in Subsection~\ref{subsec:polyhedralpair}, class $(S,U_0)^E$ contains a simple polyhedron $P_0$ with $U_0$ vertices and $S$ faces. Remember that to construct $P_0$ we started with a tetrahedron $T$ in class $(4,4)^E$, and in each step we
truncated a vertex of the polyhedron sufficiently close to this vertex. Throughout the process, the vertex can be chosen as one of those created during the previous step. 
Since in this case the conditions of Lemma~\ref{lem:face_trunc} are satisfied for any face of $P_0$, applying Lemma~\ref{lem:face_trunc}
to it we obtain a polyhedron $P_1$  with two more vertices, one more face, two more unstable and the same number of stable points.
By subsequently applying the same procedure, we can construct a convex polyhedron in class $(S,U)^E$ for every even value of $U$. To obtain a polyhedron in class  $(S, U)^E$ where $U$ is odd, we can modify a polyhedron in class $( S, U-1)^E$ according to Corollary~\ref{cor:vertex_build}.

Now, consider the case that $2 \leq S \leq 4$, and $U \geq 6$. Then, starting with a tetrahedron in class  $( S, 4)^E$ (based on the data of Table~\ref{9classes}, all three tetrahedra meet the conditions of Lemma~\ref{lem:face_trunc}) we can repeat the argument in the previous paragraph. 

\subsubsection{Case 4}  $U \geq 5$ and $S > 2U-4$, or $2 \leq U \leq 4$ and $S \geq 6$.\\

Theorem~\ref{thm:main1} in Case 4 can be deduced from Case 3 using direct geometric properties of polarity. Nevertheless, also the proof in Case 3 via Lemma~\ref{lem:face_trunc} can be dualized as well. In Lemma~\ref{lem:vertex_build} and Corollary~\ref{cor:vertex_build} we prove dual versions of Lemma~\ref{lem:face_trunc} and Corollary~\ref{cor:face_trunc}, respectively, which we are going to use also in Section~\ref{ss:monostatic}, in our investigation of monostatic polyhedra. Since Theorem~\ref{thm:main1} follows from Lemma~\ref{lem:vertex_build} and Corollary~\ref{cor:vertex_build} similarly like in the proof of Case 3, we leave it to the reader.

We start with the proof using polarity.  Considering a tetrahedron $T$ centered at $o$, a straightforward modification of the construction in Lemma~\ref{lem:inductive_step_12} and by Remark~\ref{rem:face_trunc}, we may construct a simple polyhedron $P$ with $U$ vertices that has $S$ stable and $U$ unstable equilibrium points with respect to $o$.
Using small truncations, we may assume that $P$ is arbitrarily close to $T$ measured in Hausdorff distance. Furthermore, without loss of generality, we may assume that a face of $T$, and all vertices of this face have degree $3$ in $P$. Let this face of $T$ be denoted by $F$.

Recall that $P^{\circ}$ denotes the polar of $P$. By Lemma~\ref{lem:simplices}, $c(T^{\circ}) = o$, and by the continuity of polar and the center of mass, $c(P^{\circ})$ is `close' to $o$. On the other hand, since the vertex $q$ of $P^{\circ}$ corresponding to $F$ has degree $3$, and each face containing $q$ is a triangle, Lemma~\ref{lem:Jacobian_face} implies that by a slight modification of $q$ we obtain a polyhedron $Q$ such that $c(Q) = o$, and a face/edge/vertex of $Q$ contains an equilibrium point with respect to $o$ if, and only if the corresponding vertex/edge/face of $P$ contains an equilibrium point with respect to $o$. Thus, $Q$ satisfies the required properties.

As we mentioned, an alternative way to prove Theorem~\ref{thm:main1} in Case 4 is using Lemma~\ref{lem:vertex_build} and Corollary~\ref{cor:vertex_build}.

\begin{lem}\label{lem:vertex_build}
Let $P \in (S,U)^E$ be a convex polyhedron with $f$ faces and $v$ vertices. Let $q_i , i=1...j-1,j,...m (j\geq 3)$ be successive vertices of an $m$-gonal $(m\geq 3)$ face $F$ of $P$ such that 
\begin{itemize}
 \item[i)] $P$ has a stable equilibrium point $c_F$ on $F$, which is contained in the relative interior of the triangle $T = \conv\{q_1, q_{j-1}, q_j \}$;
 \item[ii)] the edge $E = [q_{j-1} q_j]$ contains a saddle-type equilibrium point $c_{E}$; 
 \item[iii)] the vertices $q_i , i = 2, \ldots,j-1$ and $i = j+1, \ldots, m$ are trivalent; 
 \item[iv)] $q_1$, $c_F$ and $c_E$ are not collinear.
\end{itemize}
Then there exists a polyhedron $P' \in (S+2,U)^E$ with $f+2$ faces and $v+1$ vertices. 
\end{lem}

\begin{figure}[ht]
\begin{center}
\includegraphics[width=0.9\textwidth]{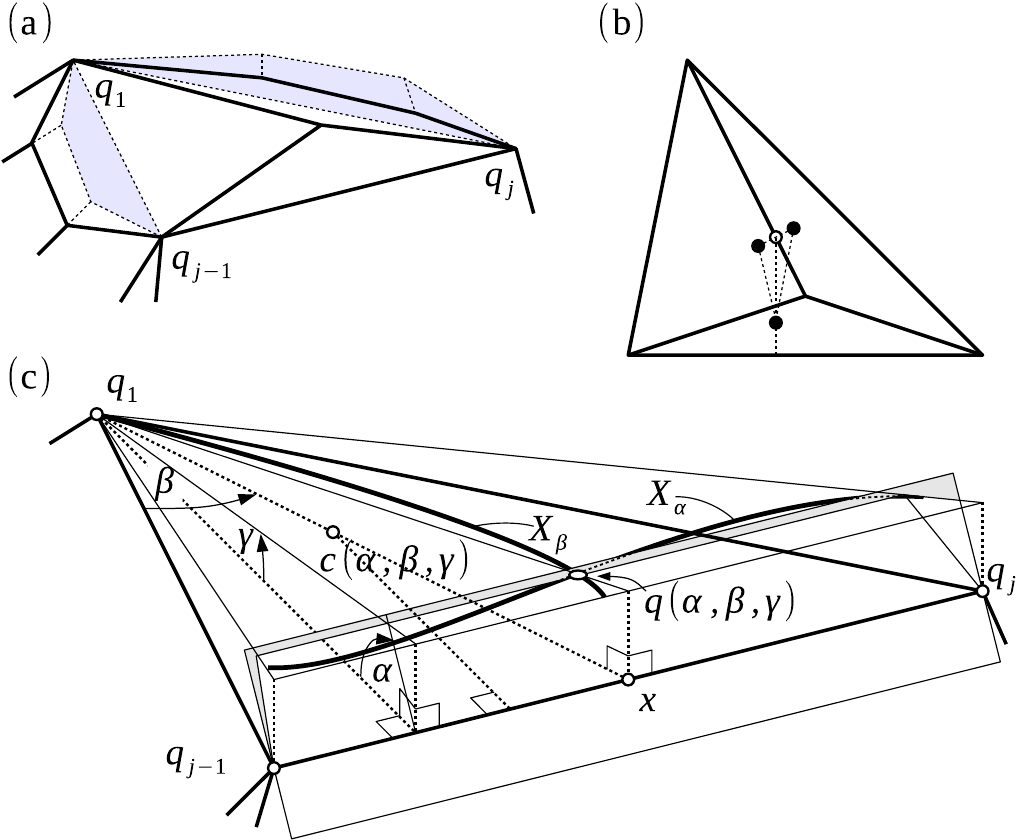}
\caption{Increasing the number of stable equilibria by two. Schematic view of the pyramid with three light faces instead of the original dark one denoted as $F$ (a); view perpendicular to face $F$: full circles mean stable equilibrium points, the empty circle is the projection of $c(\alpha,\beta,\gamma)$ onto $F$ (b); Illustration for the application of the Cube Separation Theorem for compact sets $X_\alpha$ and $X_\beta$ (c).}
\label{fig:splus2}
\end{center}
\end{figure}

\begin{proof}
In the proof, we show that for a sufficiently small pyramid erected over the triangle $T = \conv \{ q_1, q_{j-1}, q_j\}$ (which is contained by $F$ and carries a stable equilibrium point) followed by a truncation of $P$ by the plane of the three new faces of the pyramid, results in three new faces instead of $F$ all carrying stable equilibrium and two new edges both carrying saddle-type equilibrium, see Fig.~\ref{fig:splus2}a. 

Let the intersection point of the line through $q_1$ and $c_F$ with $E$ be denoted by $x$. We choose the apex $q$ of the pyramid from a fixed, sufficiently small neighborhood $V$ of $x$. Let $U$ be the set of the centers of mass of the modified convex polyhedra, which we denote by $P(q)$. 
We choose $V$ in such a way that, apart from the three new faces and edges, and the new vertex, $P(q)$ and $P$ have equilibrium points on the same faces and edges, and at the same vertices. Furthermore, we choose $V$ such that for all $q \in V$, the face structure of the resulting polyhedron $P(q)$ is the one described in the previous paragraph, and for any $y \in U$, the Euclidean distance function from $y$ on $[q_{j-1},q] \cup [q,q_j]$ has a unique local minimum, and this point is different from $q$, for all $q \in V$. Note that the latter condition implies that the new vertex $q$ is \emph{not} an unstable equilibrium point.
Thus, we need to prove only that, with a suitable choice of $q$, all the three new faces contain a new stable equilibrium point.

We parametrize $q$ using the following parameters:
\begin{itemize}
\item the angle $\alpha$ of the plane of $\conv \{ q_{j-1},q_j,q \}$ and the plane of $F$. Here we assume that $0 \leq \alpha \leq \alpha_0$, where the sum of $\alpha_0$ and the dihedral angle of $P$ at $E$ is $\pi$.
\item the angle $\beta$ between two rays, both starting at $q_1$, and containing $q_{j-1}$ and the orthogonal projection $q_F$ of $q$ onto the plane of $F$, respectively. Here we set $\beta_1 \leq \beta \leq \beta_2$, where $[\beta_1,\beta_2]$ is a sufficiently small interval containing the angle $\angle q_{j-1} q_1 c_F$.
\item the angle $\gamma$ between the ray starting at $q_1$ and containing $q$, and the plane of $F$. Here we assume that $0 < \gamma < \gamma_0$ for some small, fixed value $\gamma_0$.
\end{itemize}
We choose the values of $\beta_1, \beta_2, \gamma_0$ such that in the permitted range of the parameters, $q \in V$.
For brevity, we may refer to $P(q(\alpha,\beta,\gamma))$ as $P(\alpha,\beta,\gamma)$, $c(P(\alpha,\beta,\gamma))$ as $c(\alpha,\beta,\gamma)$ and observe that these three quantities determine $q$.

Note that, using the idea of the proof of Lemma~\ref{lem:magnitude}, we have that $|c(P(q))-c(P)| = O(\gamma)$, and for some constant $C > 0$ independent of $\alpha,\beta,\gamma$, if $|\alpha'-\alpha| \leq \gamma$, then $|c(\alpha',\beta,\gamma)-c(\alpha,\beta,\gamma)| \leq C \gamma^2$.

Fix some $\gamma > 0$, and let $X_{\alpha}$ be the set of pairs $(\alpha,\beta) \in [0,\alpha_0] \times [\beta_1,\beta_2]$ such that the planes through $E$, and containing $c(\alpha,\beta,\gamma)$ and $q(\alpha,\beta,\gamma)$, respectively, are perpendicular.
Furthermore, let $X_{\beta}$ be the set of pairs $(\alpha,\beta) \in [0,\alpha_0] \times [\beta_1,\beta_2]$ such that $q_1$, and the projections of $c(\alpha,\beta,\gamma)$ and $q(\alpha,\beta,\gamma)$ onto the plane of $F$ are collinear. If $\gamma > 0$ is sufficiently small, the property $|c(P(q))-c(P)| = O(\gamma)$ implies that $X_{\alpha}$ strictly separates the sets $\{ (0,\beta) : \beta \in [\beta_1,\beta_2]\}$ and $\{ (\alpha_0,\beta) : \beta \in [\beta_1,\beta_2]\}$, and $X_{\beta}$ strictly separates the sets $\{ (\alpha,\beta_1) : \alpha \in [0,\alpha_0]\}$ and $\{ (\alpha,\beta_2) : \alpha \in [0,\alpha_0]\}$. Since $X_{\alpha}$ and $X_{\beta}$ are compact, we may apply the Cube Separation Theorem \cite{Postnikov} as in the proof of Lemma~\ref{lem:face_trunc}. From this, it follows that there is some $(\alpha_{\gamma},\beta_{\gamma}) \in X_{\alpha} \cap X_{\beta}$.

It is easy to see that $(\alpha_{\gamma},\beta_{\gamma}) \in X_{\alpha}$ implies that for sufficiently small values $\gamma$, $P(\alpha_{\gamma},\beta_{\gamma},\gamma)$ has stable equilibrium points on both faces containing the new edge $[q_1,q(\alpha_{\gamma},\beta_{\gamma},\gamma)]$.
Furthermore, the orthogonal projection of $c(\alpha_{\gamma},\beta_{\gamma},\gamma)$ onto the plane containing $E$ and $q=q(\alpha_{\gamma},\beta_{\gamma},\gamma)$ lies on $E$. Now, let us replace $\alpha_{\gamma}$ by $\alpha' = \alpha_{\gamma} - \gamma$. Then, since in this case $|c(\alpha',\beta_{\gamma},\gamma)-c(\alpha_{\gamma},\beta_{\gamma},\gamma)| \leq C \gamma^2$, we have that if $\gamma$ is sufficiently small, then the orthogonal projection of $c(\alpha',\beta_{\gamma},\gamma)$ onto the face $\conv \{ q_{j-1},q_j,q(\alpha',\beta_{\gamma},\gamma)\}$ lies inside the face; that is, $P$ has a stable equilibrium point on this face. This yields the assertion.
\end{proof}

\begin{cor}\label{cor:vertex_build}
If all conditions (i)--(iv) of Lemma~\ref{lem:vertex_build} hold, then there is a polyhedron $P'' \in (S+1,U)^E$ with $f+2$ faces and $v+1$ vertices.
\end{cor}

\section{Monostatic polyhedra: proof of Theorem \ref{thm:SUbounds}}\label{ss:monostatic}

Our theory of mechanical complexity highlights the special role of polyhedra in the first row and first column of the $(S,U)$ grid. These objects have either only one stable equilibrium point (first row) or just one unstable equilibrium point (first column) and therefore they are called collectively \emph{monostatic}. In particular, the first row is sometimes referred to as \emph{mono-stable} and the first column as \emph{mono-unstable}. Our theory provided only a rough lower bound for their mechanical complexity. While no general upper bound is known, individual constructions provide upper bounds for some particular classes; based on these values one might think that the mechanical complexity of these classes, in particular when both $S$ and $U$ are relatively low, is very high. Monostatic objects have peculiar properties, apparently
the overall shape in these equilibrium classes is constrained. In \cite{VarkonyiDomokos} the thinness $T$ and the flatness $F$ of convex bodies is defined $(1 \leq T,F \leq \infty)$ and it is shown that, for nondegenerate convex bodies, $T=1$ if and only if $U=1$ and $F=1$ if and only if $S=1$. This  constrained overall geometry may partly account for the high mechanical complexity of monostatic polyhedra. 

\subsection{Known examples}
The first (and probably best) known such object is the monostatic polyhedron $P_C$ constructed by Conway and Guy in 1969 \cite{Conway}(cf. Figure \ref{fig:mono_2}) having mechanical complexity $C(P_C)=96$. Recently, there have been two additions: the polyhedron $P_B$ by Bezdek \cite{Bezdek} (cf. Figure \ref{fig:mono_3})
and the polyhedron $P_R$ by Reshetov \cite{Reshetov} with respective mechanical complexities $C(P_B)=64$ and $C(P_R)=70$.
It is apparent that all of these authors were primarily interested in minimizing the number of faces on the condition that there is only one stable equilibrium, so, if one seeks minimal complexity in any of these classes it is possible that these constructions could be improved. Also, as we show below, the same ideas can be used to construct examples of mono-unstable polyhedra.
The construction in \cite{Conway} relies on a delicate calculation for a  certain discretized planar spiral, defining a planar polygon $P$, serving as the basis of a prism which is truncated in an oblique manner (cf. Figure \ref{fig:mono_2}). The spiral consists of $2m$ similar right triangles, each having an angle $\beta = \pi/m$ at the point $o$. The cathetus of the smallest pair of triangles has length $r_0$, and this will be the vertical height of $o$ when the solid stands in stable equilibrium.

\begin{figure}[ht]
\begin{center}
\includegraphics[width=0.9\textwidth]{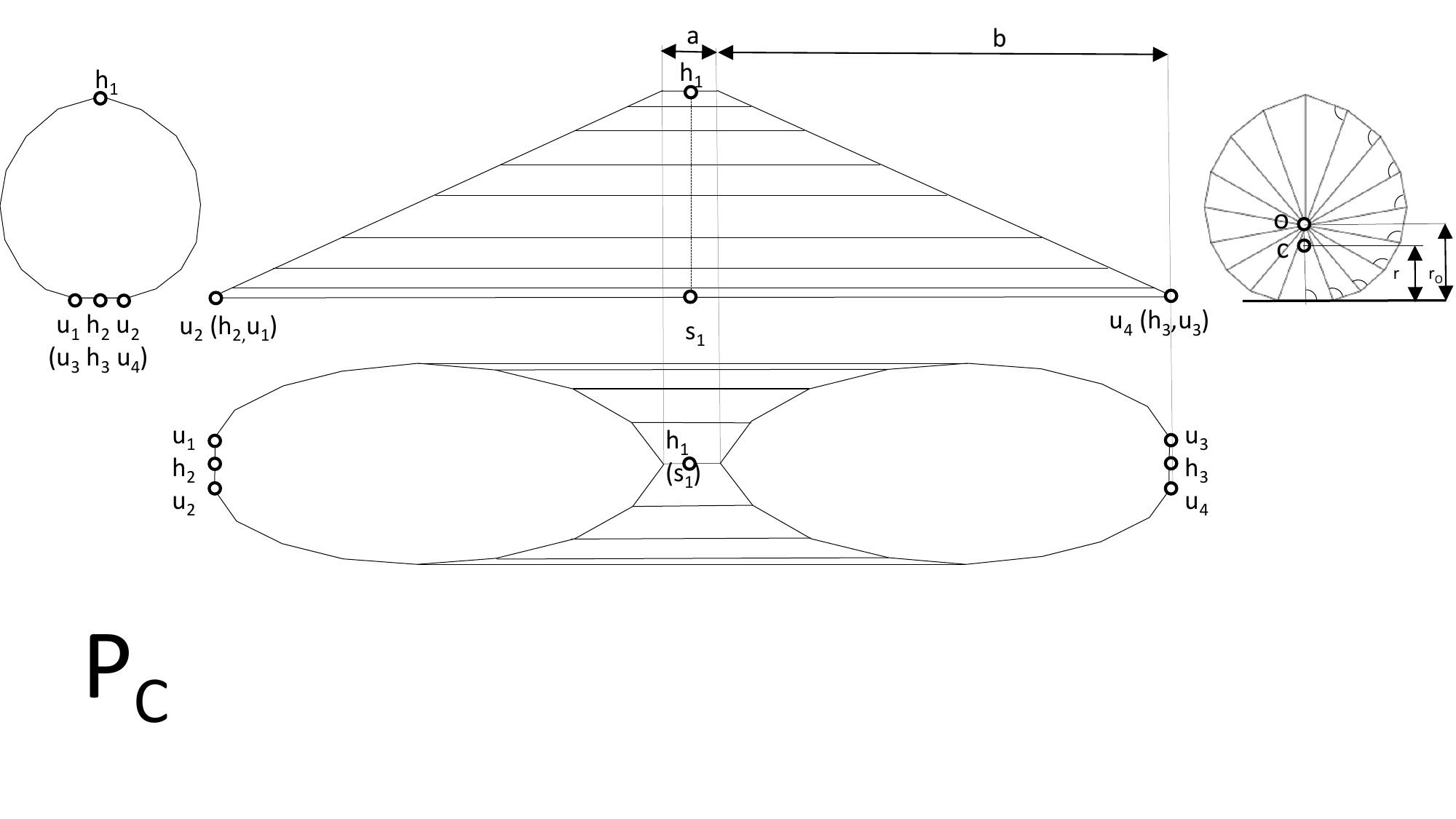}
\caption{Schematic view of the monostatic polyhedron $P_C \in (1,4)^E, (19,34)^C$  constructed by Conway and Guy in 1969 \cite{Conway}. Stable, unstable and saddle-type equilibria are marked with  $s_i, u_j,h_k$, $i=1, j=1,2,3,4, k=1,2,3$, respectively. Complexity can be computed as $C(P_C)= 2(19+34-1-4)=96$}
\label{fig:mono_2}
\end{center}
\end{figure}

\begin{figure}[ht]
\begin{center}
\includegraphics[width=0.9\textwidth]{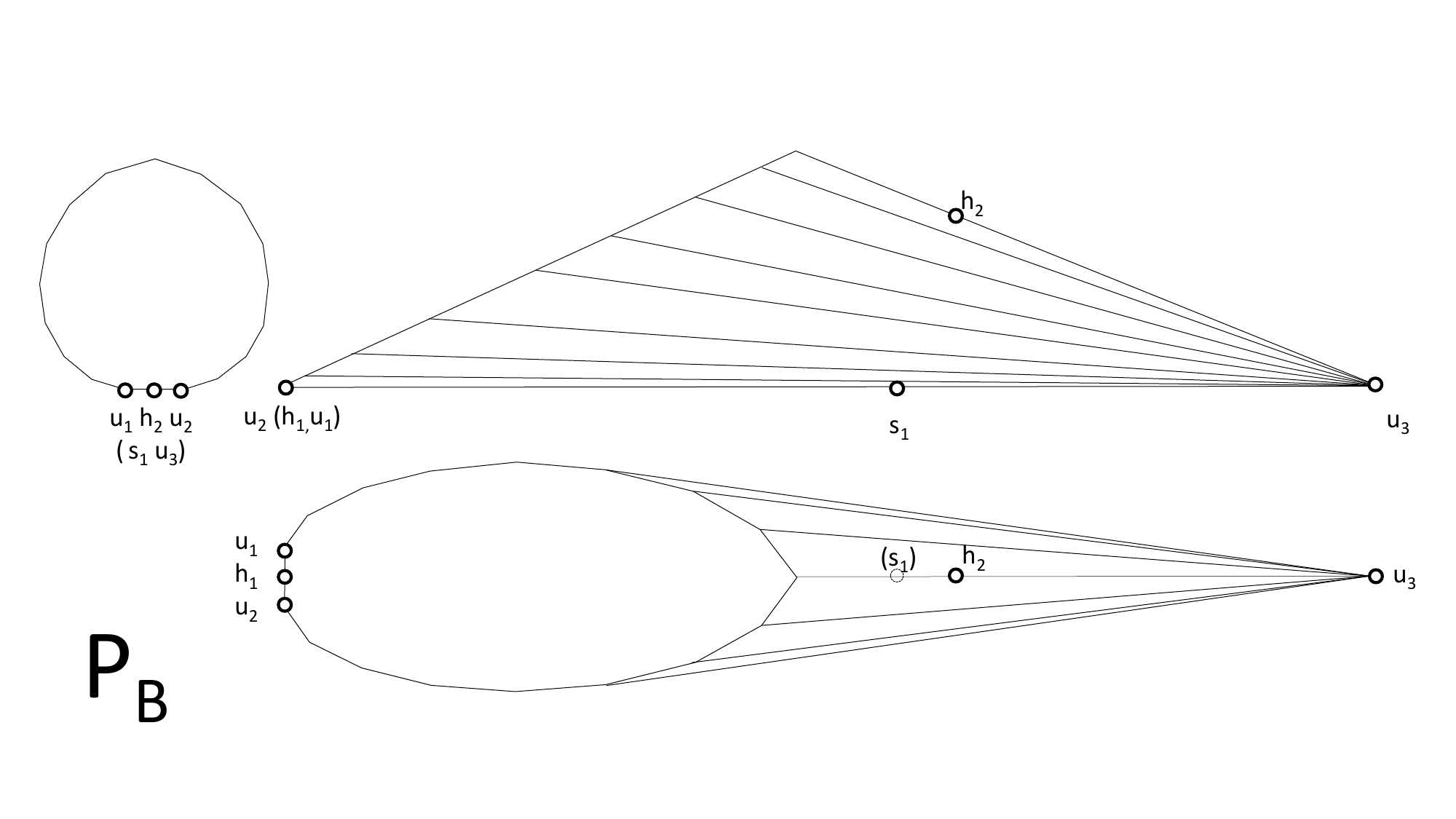}
\caption{Schematic view of the monostatic polyhedron $P_B \in  (1,3)^E, (18,18)^C$ constructed by Bezdek in 2011 \cite{Bezdek}. Stable, unstable and saddle-type equilibria are marked with  $s_i, u_j,h_k$, $i=1, j=1,2,3 k=1,2$, respectively. Complexity can be computed as 
$C(P_B)= 2(18+18-1-3)=64$}
\label{fig:mono_3}
\end{center}
\end{figure}

 We denote the height of the center of mass  $c$ by $r$ in the same configuration.
It is evident from the construction that if $P$ is a homogeneous planar disc then we have $r>r_0$ since such a disc cannot be monostatic \cite{DomokosRuina}. However, it is also clear that for a non-uniform mass distribution resulting
in $r<r_0$, $P$ would be monostatic (cf. Figure \ref{fig:mono_2}). In the construction of Conway and Guy
we can regard $r$ as a function $r(a,b)$ of the geometric parameters $a,b$ (cf. Figure \ref{fig:mono_2}).
Apparently, $r(0,b)=r_1$ and $r(a,0)=r_2$ are constants. If $P$ is the aforementioned homogeneous disc then we have  $r=r_2>r_0$. Next we state
a corollary to the main result of \cite{Conway}:
\begin{cor}\label{conway}
If $m \geq 9$ then $r_1<r_0$.
\end{cor}

\subsection{Examples in $(3,1)^E$ and $(2,1)^E$}

Consider a Conway construction with $ b = 0$ and denote its vertical  centroidal coordinate by $r_3$: it equals the centroidal coordinate of a plane polygon depicted on the right of Fig.~\ref{fig:mono_2}. Now erect a mirror-symmetric pyramid over the polygon with its apex close to the bottom edge: the vertical coordinate of the body centre of the pyramid will then be close to $3r_3/4$. It can be shown that for a sufficiently flat pyramid (we call it $P_3$) will be in classes $(3,1)^E$ and $(18,18)^C$. Introducing a small asymmetry to $P_3$ by moving the apex off the symmetry plane, a polyhedron $P_2$ is obtained which belongs to classes $(2,1)^E$ and $(18,18)^C$.    

These 'mono-unstable' polyhedra are illustrated in Figure \ref{fig:mono_4}.
An overview of the discussed monostatic polyhedra is shown in Figure \ref{fig:mono_1} on an overlay of the  $(f,v)^C$ and $(S,U)^E$ grids.

\begin{figure}[ht]
\begin{center}
\includegraphics[width=0.9\textwidth]{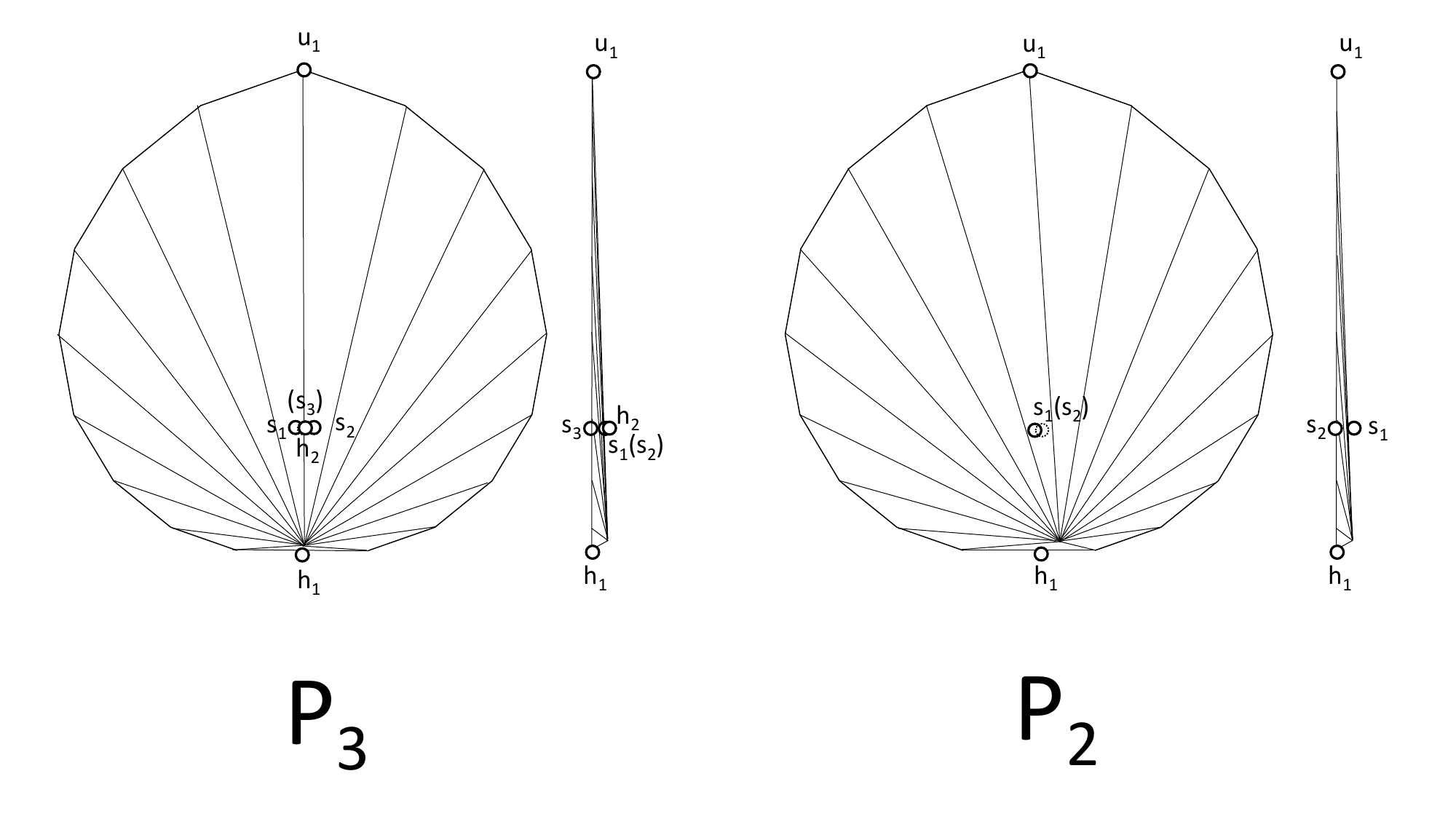}
\caption{Schematic view of two polyhedra $P_3 \in (3,1)^E, (18,18)^C$ and $P_2 \in  (2,1)^E, (18,18)^C$, obtained by using the ideas of the Conway and Bezdek constructions. Stable, unstable and saddle-type equilibria are marked with  $s_i, u_j,h_k$. In case of $P_3$ we have $i=1,2,3, j=1, k=1,2$ and in case of $P_2$ we have $i=1,2, j=1, k=1$. Complexity can be computed as $ C(P_3)=2( 18+18-3-1)=64, C(P_2)= 2(18+18-2-1)=66 $
}
\label{fig:mono_4}
\end{center}
\end{figure}

\subsection{Proof of Theorem \ref{thm:SUbounds}}

\begin{proof}
Consider the case $C(1,U)$ first. The polyhedron $P_C$ has a narrow rectangular face with a stable point and two saddle points on opposite short edges of the same face. They do not satisfiy condition (i) of Lemma~\ref{lem:face_trunc} because of being collinear, but  both 17-gonal faces of $P_C$ can slightly be rotated to get $P'_C$ according to Remark~\ref{rem:smallperturb} in a way that no equilibrium points appear or disappear but the two edges with saddle points become nonparallel, and thus Lemma~\ref{lem:face_trunc} turns to be applicable. 

Since the same face of $P_C$ contains four unstable points as well (and none of them is collinear with the stable and any saddle point), Corollary~\ref{cor:face_trunc} can directly be applied to get $P_D$ with $C(P_D)= C(P_C)+2 = 98$. It means that $C(1,4) \leq 2R(1,4)+90$ and $C(1,5) \leq 2R(1,5)+90$. Applying now Lemma~\ref{lem:face_trunc} on both $P_C$ and $P_D$ successively, the assertion readily follows.
Note that $P_B$ could not be used as departure instead of $P_C$, since its saddle points are not on edges of the same face.

A similar path is taken for the case $C(S,1)$. Depart now $P_3$ with $C(P_3) = 64$: that polyhedron has a 17-gonal face with a stable equilibrium and there is a vertex and an edge on its perimeter having an unstable and a saddle point, respectively. Now it is possible again to slightly rotate the plane of the symmetric triangular face about an axis which is perpendicular to the 17-gon and runs through the apex of the pyramid, making the stable ($s_3$) and saddle ($h_1$) point to move off the symmetry axis of the $17$-gon, so that they become non-collinear with $u_1$ (Remark~\ref{rem:smallperturb} guarantees that it can always be done without changing the number of equilibrium points of any kind). Applying or not Corollary~\ref{cor:vertex_build} first then Lemma~\ref{lem:vertex_build} successively gives $C(S,1) \leq S+ 61$ and $C(S,1) \leq S+62$ for odd and even $S$, respectively, which is equivalent to the second statement of the theorem.
\end{proof}

\subsection{G\"omb\"ocedron prize}
While the construction of monostatic polyhedra with less than 34 edges appears to be challenging  (cf. Figure \ref{fig:mono_1}), the only case which has been excluded is the tetrahedron with $e=6$ edges.

\begin{figure}[ht]
\begin{center}
\includegraphics[width=0.85\textwidth]{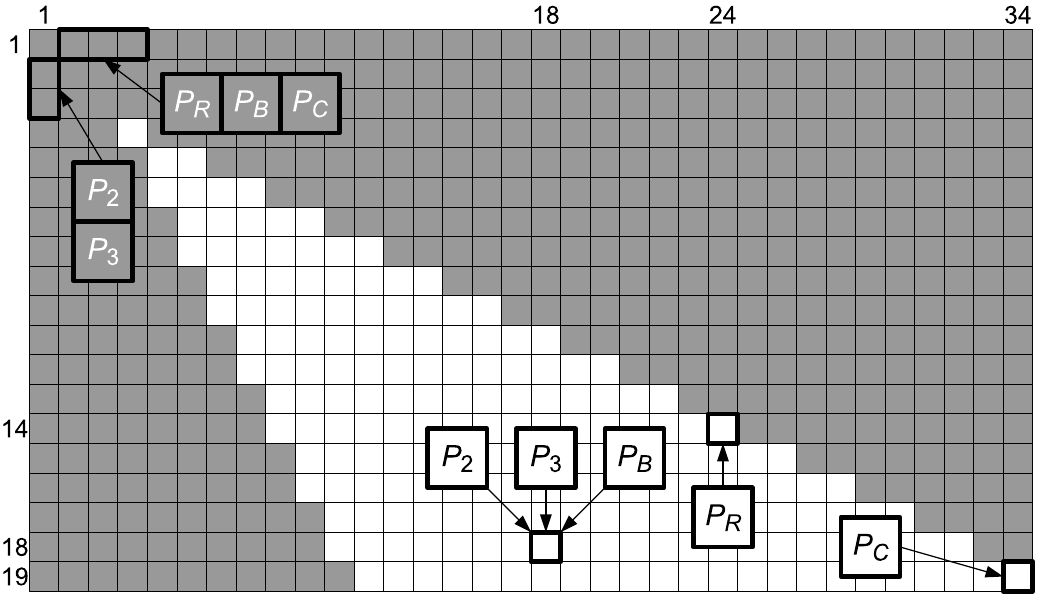}
\caption{Polyhedra with a single stable or unstable equilibrium point. 
The grid shown is an overlay of the  $(f,v)$ and the $(S,U)$ grids. White squares correspond to polyhedral pairs. Location of monostatic polyhedra is shown with black capital letters on the $(f,v)$ grid and white capital letters 
on the  $(S,U)$ grid. Abbreviations:
$P_C$: Conway and Guy, 1969 \cite{Conway}, $P_B$: Bezdek, 2011 \cite{Bezdek}, $P_R$: Reshetov, 2014 \cite{Reshetov}. $P_2,P_3$: current paper, Figure \ref{fig:mono_4}.
Complexity for these polyhedra can be readily
computed as $C(P_C)=96, C(P_B)=64, C(P_R)=70, C(P_3)=64, C(P_2)=66$.}
\label{fig:mono_1}
\end{center}
\end{figure} It also appears to be very likely that G\"omb\"oc-like polyhedra in class  $(1,1)^E$ do exist, however, based on this chart and the previous results, one would expect polyhedra with high mechanical complexity. To further motivate this research we offer a prize for establishing the mechanical complexity $C(1,1)$, the amount $p$ of the prize is given in US dollars as
\begin{equation}
p=\frac{10^6}{C(1,1)}.
\end{equation}

\section{Generalizations and applications}\label{sec:conclusions}

\subsection{Complexity of secondary equilibrium classes}

A special case of Theorem \ref{thm:main1} states that for any polyhedral pair $(f,v)$
 one can construct a homogeneous polyhedron $P$ with $f$ faces and $v$ vertices
in such a manner that $C(P)=0$. In other words, in any primary combinatorial class there exist polyhedra with zero
complexity.
A natural generalization of this statement is to ask whether this is also true for any \emph{secondary} combinatorial
class of convex polyhedra. While we do not have this result,  we present an affirmative statement for
the inhomogeneous case:

\begin{prop}\label{prop:Koebe}
Let $P$ be a Koebe polyhedron, i.e. a convex polyhedron midscribed (edge-circumscribed) about the unit sphere $\S^2$
with center $o$.  Then every face, edge and vertex of $P$ carries an equilibrium point with respect to $o$.
\end{prop}

\begin{proof}
 By (\ref{Poincare}) and (\ref{Euler}) it is sufficient to show that every edge of $P$ contains an equilibrium point with respect to $o$.

Let $E$ be an edge of $P$ that touches $\S^2$ at a point $q$, and let $H$ be the plane touching $\S^2$ at $q$. Clearly, $H$ is orthogonal to $q$, and since every face of $P$ intersects the interior of the sphere, we have $H \cap P = E$. Thus, $q$ is an equilibrium point of $P$ with respect to $o$.
\end{proof}

 Since a variant of the Circle Packing Theorem \cite{Brightwell} states that every combinatorial class contains a Koebe polyhedron, it follows that every combinatorial class contains an inhomogeneous polyhedron with zero mechanical complexity.
To find a homogeneous representative appears to be a challenge.

 In \cite{Langi}, the author strengthened the result in \cite{Brightwell} by showing the existence of a Koebe polyhedron $P$ in each combinatorial class such that the center of mass of the $k$-dimensional skeleton of $P$, where $k=0,1$ or $2$, coincides with $o$. This result and Proposition~\ref{prop:Koebe} imply that replacing $c(P)$ by the center of mass of the $k$-skeleton of a polyhedron with $0 \leq k \leq 2$, every combinatorial class contains a polyhedron with zero mechanical complexity.

\subsection{Inverse type questions}
The basic goal of this paper is to explore the nontrivial links between the  combinatorial $(f,v)^C$ and the mechanical $(S,U)^E$ classification of convex polyhedra.  The concept of mechanical complexity (Definition \ref{defn:complex}) helps to  explore the  $(S,U)^E \to (f,v)^C$ direction of this link. Inverse type questions may be equally useful to understand this relationship: for example,  a natural question to ask is the following: 
Is it true that any equilibrium class  $(S,U)^E$ intersects all but at most finitely many combinatorial classes $(f,v)^C$? Here it is worth noting that it is easy to carry out local deformations on a polyhedron that increase the number of faces and vertices, but not the number of equilibria. Alternatively, one may ask to
provide the list of all  $(S,U)^E$ classes represented by homogeneous polyhedra in a given combinatorial class $(f,v)^C$. A similar question may be asked for a secondary combinatorial class of polyhedra. In general, we know little about the answers, however we certainly know that (\ref{trivbound}) holds
and we also know that $S=f,U=v$ is a part of this list. The minimal values for $S$ and $U$ are less clear. In particular, based on our previous results it appears that the values $S=1$ and $U=1$ can be
only achieved for sufficiently high values of $f,v$. On the other hand, Theorem~\ref{thm:tetra} and Lemma \ref{lem:complexity_tetrahedron} resolve this problem at least for the  $(4,4)^C$ class.
The latter is based on a global numerical search and this could be done at least for some polyhedral classes, although the computational time grows with exponent $(f+v)$.

\subsection{Inhomogeneity and higher dimensions}
While here we described only 3D shapes, the generalization of Definitions \ref{defn:complex} and \ref{defn:class} to arbitrary dimensions is straightforward. While the actual values of mechanical complexity  are trivial in the planar case (class $(2)^E$ has mechanical complexity 2 and every other equilibrium class has mechanical complexity zero), the $d>3$ dimensional case appears an interesting question in the light of the results of Dawson et al. on monostatic simplices in higher dimensions \cite{Dawson, Dawson2, Dawson3}.
We formulated all our results for homogeneous polyhedra, nevertheless, some remain valid in the inhomogeneous case which
also offers interesting open questions. In particular,
the universal lower bound (\ref{trivbound}) is independent of the material density distribution
so it remains valid for inhomogeneous polyhedra and as a consequence,
so does Theorem \ref{thm:main1}. However, our other results (in particular the bounds for monostatic
equilibrium classes) are only valid for the homogeneous case. In the latter context it is interesting to
note that Conway proved the existence of inhomogeneous, monostatic tetrahedra \cite{Dawson}.

\subsection{Applications}

Here we describe some problems in mineralogy, geomorphology and industry where the concept of mechanical complexity could potentially contribute to the efficient description and the better understanding of the main phenomena.

\subsubsection{Crystal shapes}
Crystal shapes are probably the best known examples of polyhedra appearing in Nature and the literature on their morphological,
combinatorial and topological classification is substantial \cite{Klein}. However, as crystals are not just geometric objects but also (nearly homogeneous) 3D solids, their equilibrium classification appears to be relevant. The number of static balance points has been recognized as a meaningful geophysical shape descriptor \cite {Natural, Pebbles, Williams} and it has also been investigated in the context of crystal shapes \cite{ceug}. The theory outlined in our paper may help to add new aspects to their understanding. While the study of a broader class of crystal  shapes is beyond the scope of this paper, we can illustrate this idea in Figure \ref{fig:crystal} by two examples of quartz crystals with identical number of faces displaying a large difference in mechanical complexity. The length $a$ of the middle, prismatic part of 
the hexagonal crystal shape (appearing on the left side of Figure \ref{fig:crystal}) is not fixed in the crystal.
As we can observe,
for sufficiently small values of the length $a$ the crystal will be still in the same combinatorial class $(18,14)^C$, however, its
mechanical complexity will be reduced to zero.

\begin{figure}[ht]
\begin{center}
\includegraphics[width=0.7\textwidth]{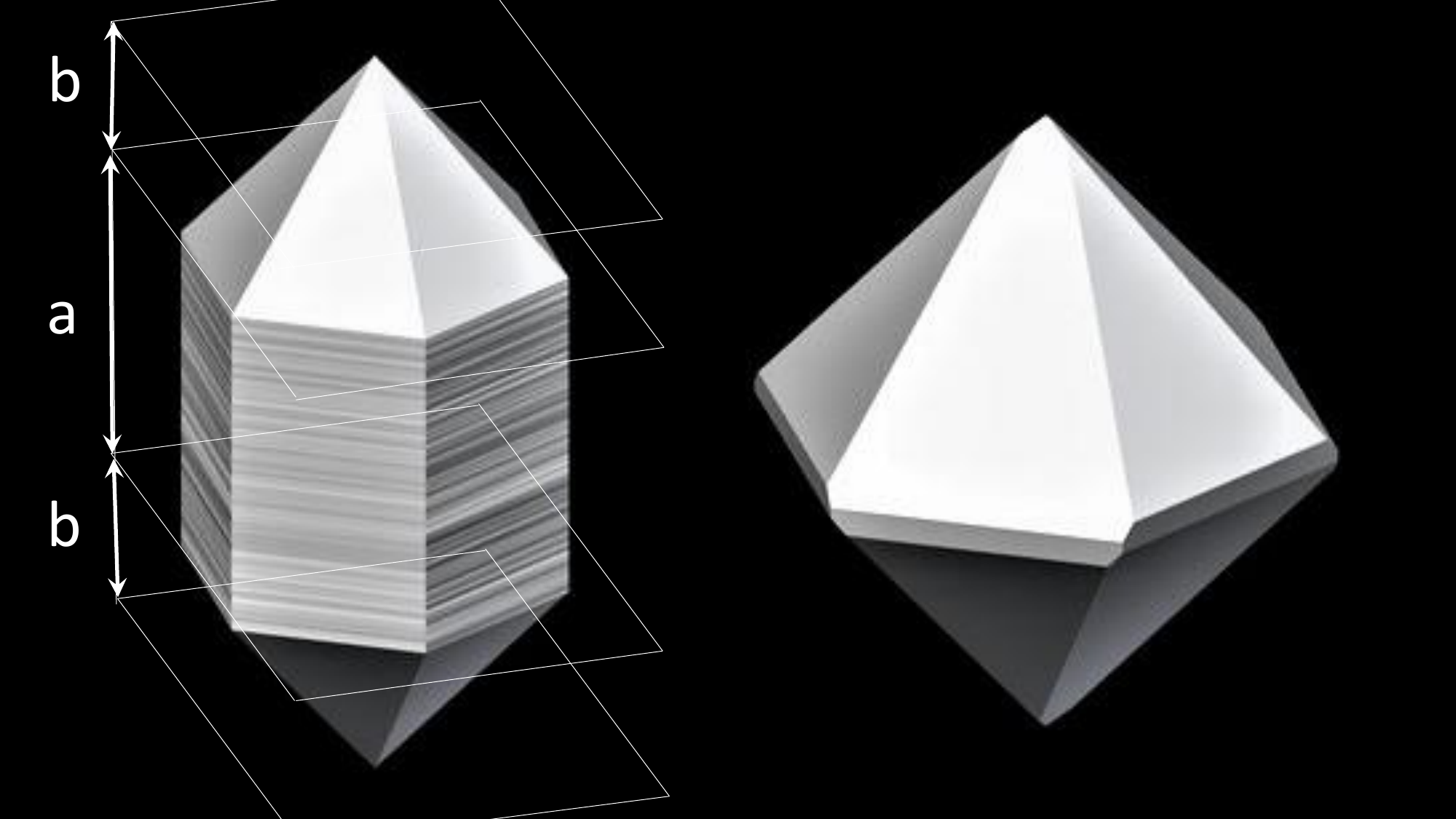}
\caption{Quartz crystals. Left: Hexagonal habit in classes $(18,14)^C$ and $(6,2)^E$, $C(P)=48$. Right: Cumberland habit \cite{White} in classes$(18,32)^C$ and $(12,8)^E$, $C(P)=60$. Picture source \cite{minerals}.}
\label{fig:crystal}
\end{center}
\end{figure}

\subsubsection{Random polytopes, chipping models and natural fragments}
There is substantial literature on the shape of random polytopes \cite{Schneider} which are obtained by successive intersections
of planes at random positions. Under rather general assumptions on the distribution of the intersecting planes it can be shown
\cite{Schneider} that the \emph{expected} primary combinatorial class of such a random polytope is  $(6,8)^C$, however, there
are no results on the mechanical complexity. A very special limit of random polytopes can be created if we use a \emph{chipping model} \cite{Sipos, Krapivsky} where one polytope is truncated with planes in such a manner that the truncated pieces
are small compared to the polytope. Although not much is known about the combinatorial properties of these polytopes, it can be shown
\cite{DomokosLangi} that under a sufficiently small truncation the mechanical complexity either remains constant or it increases (this is illustrated in Figure \ref{fig:1}). 
Apparently, random polytopes can be used to approximate natural fragments \cite{Kun}. There is data 
available on the number and type of static equilibria of the latter, so any result on the mechanical complexity of random polytopes could be readily tested and also used to identify fragmentation processes.

\subsubsection{Assembly processes}
In industrial assembly processes parts are processed by a feeder and often these parts can be approximated by polyhedra. 
These polyhedra arrive in a random orientation on a horizontal surface (tray) and end up ultimately on one of their faces carrying a stable equilibrium. Based on the relative frequency of this position, one can derive \emph{face statistics} and the throughput of a part feeder is heavily influenced by the face statistics of the parts processed by the feeder. Design algorithms for feeders 
are often investigated from this perspective \cite{Bohringer, Varkonyi}.  It is apparent that one
key factor determining the entropy of the face statistics is the mechanical complexity of the polyhedron, in particular, higher
mechanical complexity leads to better predictability of the assembly process so this concept may add a useful
aspect to the description of this industrial problem.

\subsection{Concluding remarks}
We showed an elementary connection between the Euler and Poincar\'e-Hopf formulae (\ref{Poincare}) and (\ref{Euler}): the mechanical complexity of a polyhedron is determined jointly by its equilibrium class  $(S,U)^E$ and combinatorial class $(f,v)^C$. Mechanical complexity appears to be a good tool to highlight the special properties of monostatic polyhedra and offers a new approach to the classification of crystal shapes. We defined polyhedral pairs  $(x,y)$ of integers (cf. Definition \ref{def:pair}) and showed that they play a central role in both classifications: they define all possible  combinatorial classes $(f,v)^C$ while in the mechanical classification they correspond to classes with zero complexity.

\end{document}